\documentclass[a4paper,12pt]{amsart}

\usepackage{amsthm, amsmath, latexsym, amsfonts, epsfig}
\usepackage[all]{xy}
\usepackage{amssymb}
\usepackage{mathrsfs}
\usepackage{stmaryrd}
\usepackage{enumerate}
\usepackage{epsfig}

\newtheorem{Lemma}{Lemma}[section]
\newtheorem{Proposition}[Lemma]{Proposition}
\newtheorem{Theorem}[Lemma]{Theorem}
\newtheorem{Fact}[Lemma]{Theorem}
\newtheorem{Remark}[Lemma]{{Remark}}
\newtheorem{Corollary}[Lemma]{Corollary}
\newtheorem{Definition}[Lemma]{Definition}
\newtheorem{Question}[Lemma]{Question}

\newtheorem{Warning}[Lemma]{Warning}

\theoremstyle{definition}
\newtheorem{Notation}[Lemma]{Notation}
\newtheorem{nota}[Lemma]{}

\newcommand{\w}{\widetilde}
\newcommand{\ov}{\overline}
\newcommand{\un}{\underline}

\newcommand{\Spf}{\mathrm{Spf}}

\newcommand{\Pic}{\mathrm{Pic}}
\newcommand{\NS}{\mathrm{NS}}

\newcommand{\ch}{\mathrm{ch}}
\newcommand{\Td}{\mathrm{Td}}

\newcommand{\Ker}{\mathrm{Ker}}

\newcommand{\Ext}{\mathrm{Ext}}
\newcommand{\Hom}{\mathrm{Hom}}

\newcommand{\Aut}{\mathrm{Aut}}

\newcommand{\id}{\mathrm{id}}

\newcommand{\Def}{\mathrm{Def}}

\newcommand{\GL}{\mathrm{GL}}
\newcommand{\Gm}{\mathbb{G}_m}

\newcommand{\Z}{\mathbb{Z}}
\def\X{\mathcal X}

\newcommand{\Q}{\mathbb{Q}}
\newcommand{\C}{\mathbb{C}}
\renewcommand{\O}{\mathcal{O}}

\renewcommand{\P}{\mathbb{P}}
\newcommand{\PP}{\mathcal{P}}
\newcommand{\LL}{\mathcal{L}}
\newcommand{\CC}{\mathcal{C}}
\newcommand{\EE}{\mathcal{E}}
\newcommand{\OO}{\mathcal{O}}
\newcommand{\F}{\mathcal F}

\renewcommand{\L}{\mathbb{L}}

\newcommand{\II}{\mathbb{I}}

\newcommand{\Mgbar}{\overline{\mathcal{M}}_g}
\newcommand{\Mguniv}{\overline{\mathcal M}_{g,1}}
\newcommand{\Mg}{\mathcal{M}_g}
\newcommand{\mgbar}{\overline{M}_g}
\newcommand{\mg}{M_g}

\newcommand{\pdgbar}{\overline{P}_{d,g}}
\newcommand{\pdg}{P_{d,g}}
\newcommand{\pdst}{\ov{P}_{d,g}^{\rm st}}

\newcommand{\Pdgbar}{\overline{\mathcal Pic}_{d,g}}
\newcommand{\Pdg}{{\mathcal Pic}_{d,g}}
\newcommand{\Pdguniv}{\ov{\mathcal Pic}_{d,g,1}}

\newcommand{\pdstack}{\ov{\mathcal Pic}_{d,g}}

\def\pdb{\overline{\mathcal Pic}_{d,g}}
\def \gd {{\mathcal Pic}_{d,g}}

\newcommand{\can}{K_{\pdgbar}}

\title[On the birational geometry of $\pdgbar$]
{On the birational geometry \\ of the universal Picard variety}

\author{Gilberto Bini, Claudio Fontanari, Filippo Viviani}

\email{gilberto.bini@unimi.it} \curraddr{{\sc Dipartimento di Matematica \\ Universit\`a degli Studi di Milano \\ Via C. Saldini 50 \\ 20133 Milano \\ Italy.}}

\email{fontanar@science.unitn.it}\curraddr{{\sc Dipartimento di Matematica \\  Universit\`a degli Studi di Trento\\ Via Sommarive 14 \\ 38123 Trento \\ Italy.}}

\email{viviani@mat.uniroma3.it} \curraddr{{\sc Dipartimento di Matematica \\ Universit\`a Roma Tre \\ Largo S. Leonardo Murialdo 1 \\ 00146 Roma \\ Italy.}}

\thanks{The first named author has been partially supported by "FIRST" Universit\`{a} di Milano and by MIUR Cofin 2008 - Variet\`{a} algebriche: geometria, aritmetica, strutture di Hodge. The second named author has been partially supported by GNSAGA of INdAM and by MIUR Cofin 2008 - Geometria delle variet\`{a} algebriche e dei loro spazi di moduli. The third named author has been partially supported by FCT-Ci\^encia2008 and by MIUR Cofin 2008 - Geometria delle variet\`{a} algebriche e dei loro spazi di moduli.}



\begin{document}

\begin{abstract}
We compute the Kodaira dimension of the universal Picard variety $\pdg$ parameterizing line bundles of degree $d$ on curves of genus $g$ under the assumption that $(d-g+1,2g-2)=1$.
We also give partial results for arbitrary degrees $d$ and we investigate for which degrees the universal Picard varieties
are birational.
\end{abstract}

\maketitle

\section{Introduction}

The study of the birational geometry of the moduli spaces has become a very active research area after the unexpected
result of Harris-Mumford-Eisenbud (\cite{HarMum}, \cite{EH}) that the moduli space $M_g$ of curves
of genus $g$ is a variety of general type for $g \ge 24$, contradicting a long-standing conjecture of Severi on the
unirationality of moduli of curves. More recently, also the birational geometry of other moduli spaces has been
widely investigated: the moduli space of pointed curves (\cite{Log}, \cite{CF}), the moduli space
of  Prym varieties (\cite{FL}), the moduli space of spin curves (\cite{Lud2}, \cite{Far}, \cite{FV}),
to mention at least some contributions in this area.

The aim of the present paper is to investigate the birational geometry of the universal Picard variety
$$
P_{d,g} \to M_g,
$$
parameterizing smooth curves of genus $g$ together with a line bundle of degree $d$.
The following result is due to Verra (see \cite[Thm. 1.2]{Verra}).

\begin{Theorem}\emph{(\cite{Verra})}\label{Ver-thm}
The variety $P_{d,g}$ is unirational for $g\leq 9$ and any $d$.
%In particular, the Kodaira dimension $\kappa(P_{d,g})$ of $P_{d,g}$
%is equal to $-\infty$ if $g\leq 9$.
\end{Theorem}

Our main result is the computation of the Kodaira dimension of $P_{d,g}$ with $g\geq 10$ under a technical assumption on
the degree $d$. Recall that, since $P_{d,g}$ is singular and not projective, the Kodaira dimension of $P_{d,g}$, which we denote by $\kappa(P_{d,g})$,
 is defined as the Kodaira dimension of any smooth projective model of it (see \cite[Example 2.1.5]{Laz}). The previous result of Verra implies that
 $\kappa(P_{d,g})=-\infty$ for $g\leq 9$ and any $d$.

\begin{Theorem}\label{Kod-geom}
Assume that $(d-g+1,2g-2)=1$ and $g\geq 10$. The Kodaira dimension of $\pdg$ is equal to
$$\kappa(\pdg)=\begin{cases}
%-\infty &\text{ if } g\leq 9,\\
0 & \text{ if } g=10,\\
19 & \text{ if } g=11,\\
3g-3 & \text{ if } g\geq 12.
\end{cases}
$$
\end{Theorem}

In Propositions \ref{dim-Kod11} and \ref{dim-Kod}, we also determine
the Iitaka fibration (see \cite[Def. 1.3.6]{Laz}) of $\pdg$ in the non-trivial cases, namely for $g\geq 11$.
%Note that $\pdg$ was previously known to be unirational for $g\leq 9$ and any $d$ by work of Verra
%(\cite[Thm. 1.2]{Verra}).
Without any assumption on the degree $d$, we obtain the following partial result:

\begin{Theorem}\label{kod-nongeo}
The Kodaira dimension of $\pdg$ (for $g\geq 10$) satisfies the following inequalities
$$\kappa(\pdg)\leq \begin{cases}
%-\infty &\text{ if } g\leq 9,\\
0 & \text{ if } g=10,\\
19 & \text{ if } g=11,\\
3g-3 & \text{ if } g\geq 12.
\end{cases}
$$
Moreover, $\kappa(\pdg)=3g-3$ if $\kappa(\mg)\geq 0$ (and in particular for $g\geq 22$).
\end{Theorem}

Let us now explain the strategy that we use to prove the above results.
The main tool we use is the GIT compactification constructed by Caporaso
(see \cite{Cap})
$$
\phi_d:\pdgbar \to \mgbar
$$
of $P_{d,g}$  over the Deligne-Mumford moduli space $\mgbar$ of stable curves of genus $g$.
The projective normal variety $\pdgbar$ is a good moduli space for the stack
$\Pdgbar$ (see \cite{Capbis} and \cite{Melo}), whose section over a scheme
$S$ is the groupoid $\Pdgbar(S)$ of families of quasistable
curves of genus $g$
$$
f: (\mathcal{C}, \mathcal{L}) \to S
$$
endowed with a balanced line bundle $\LL$ of degree $d$ (see \ref{def-stack-scheme} for details).
Furthermore, $\pdgbar$ is a coarse moduli scheme for $\Pdgbar$ if and only
if $(d-g+1,2g-2)=1$, which is precisely the numerical hypothesis on the degree
$d$ in Theorem \ref{Kod-geom}.

Albeit $\pdgbar$ is singular, we can prove (under the same assumption on the degree) that $\pdgbar$ has canonical singularities and therefore pluricanonical forms on the smooth locus lift to any desingularization:

\begin{Theorem}\label{conj-sing}
Assume that $(d-g+1,2g-2)=1$ and that $g\geq 4$.
Then $\pdgbar$ has canonical singularities.
In particular, if $\w{\pdgbar}$ is a resolution of singularities of $\pdgbar$, then
every pluricanonical form defined on the smooth locus $\pdgbar^{\rm reg}$
of $\pdgbar$ extends holomorphically to $\w{\pdgbar}$, that is,
for all  integers $m$ we have
$$h^0(\pdgbar^{\rm reg},m K_{\pdgbar^{\rm reg}})=h^0(\w{\pdgbar}, m K_{\w{\pdg}}).
$$
\end{Theorem}
The proof of this theorem is given in Section \ref{Sec-sing}.
The restriction on the degree $d$ comes from the fact that $\pdgbar$ has finite quotient singularities if and only if $(d-g+1,2g-2)=1$; hence
only for such degrees $d$ we can apply
the Reid--Tai criterion for the canonicity of finite quotient singularities (see e.g. \cite[p. 27]{HarMum}
or \cite[Thm. 4.1.11]{Lud}).
% in the generalized form given in \cite[Appendix 1 to \S 1]{HarMum})
%and the fact that $\pdgbar$ has finite quotient singularities if and only if $(d-g+1,2g-2)=1$.
Indeed, we establish in Theorem \ref{sing-Pdst} a similar statement without any restriction
on $d$ for the open subset $\pdst\subset \pdgbar$ of
GIT-stable points of $\pdgbar$, which coincides with $\pdgbar$ if and only if $(d-g+1,2g-2)=1$.
In proving  Theorem \ref{sing-Pdst} (from which Theorem \ref{conj-sing} follows), we
determine the non-smooth locus of $\pdst$ in Proposition \ref{sing}.
%also find two results that may be interesting in their own:
%we determine the non-smooth locus of $\pdst$ in Proposition \ref{sing}
%and the locus of non-canonical singularities of $\pdst$ in Proposition \ref{can-sing}.
Note that a proof of Theorem \ref{conj-sing} for all degrees $d$ would imply the validity of
Theorem \ref{Kod-geom} without any assumptions on the degree $d$.

The above Theorem \ref{conj-sing} is crucial for our purposes because it allows us to compute the Kodaira dimension
of $\pdgbar$ as the Iitaka dimension (see \cite[Def. 2.1.3]{Laz})
of the canonical divisor $K_{\pdgbar}$ on the modular variety $\pdgbar$, instead of working on some (a priori non modular)
desingularization of $\pdgbar$. The  class of $K_{\pdgbar}$ is given by the following

\begin{Theorem} \label{K}
For any $g\geq 4$, we have
$$
K_{\pdgbar}= \phi_d^*(14\lambda-2\delta),
$$
where $\lambda, \delta$ denote the Hodge and the boundary class on $\mgbar$,
respectively.
\end{Theorem}

The proof of this theorem is given in Section \ref{canonical}. We first
compute in Theorem \ref{can-stack} the canonical class of $\Pdgbar$
through a careful application
of the Grothendieck-Riemann-Roch theorem to the universal family over $\Pdgbar$.
Then we show that the pull-back of $K_{\pdgbar}$ via the canonical map $p:\Pdgbar\to \pdgbar$
is equal to $K_{\Pdgbar}$. Note that this is in contrast with what happens
for $\mgbar$ (or for the moduli space of Prym or spin curves), where the
pull-back of the canonical class of the coarse moduli space is equal to the
canonical class of the moduli stack plus some (small) corrections at the boundary.

Theorem \ref{K} allows us to compute the Iitaka dimension of $K_{\pdgbar}$
as the Iitaka dimension of the divisor $14\lambda-2\delta$ on $\mgbar$ (because $\phi_d$
is a regular fibration).
By exploiting the rich available knowledge on the birational geometry of $\mgbar$, we
prove the following

\begin{Theorem}\label{Iitaka-Mg}
The Iitaka dimension of $K_{\pdgbar}$ on $\pdgbar$ is equal to
$$\kappa(K_{\pdgbar})=\kappa(14\lambda-2\delta)=\begin{cases}
-\infty &\text{ if } g\leq 9,\\
0 & \text{ if } g=10,\\
19 & \text{ if } g=11,\\
3g-3 & \text{ if } g\geq 12.
\end{cases}
$$
\end{Theorem}
The proof of the above theorem is given in Section \ref{Iitaka} by combining Propositions \ref{kod-infty}, \ref{dim-Kod},
\ref{dim-Kod10}, \ref{dim-Kod11}.

With the above results it is now easy to prove Theorems \ref{Kod-geom} and \ref{kod-nongeo}.
Indeed, note that we have always the inequality
\begin{equation}\label{comp-Kod}
\kappa(\pdg)\leq \kappa(\can),
\end{equation}
with equality if $(d-g+1,2g-2)=1$ by Theorem \ref{conj-sing}. From (\ref{comp-Kod}) and Theorem \ref{Iitaka-Mg}, we deduce
Theorem \ref{Kod-geom} and the first part of Theorem \ref{kod-nongeo}. The second part of Theorem \ref{kod-nongeo}
follows from Proposition \ref{kod-fib}, which is proved in Section \ref{Fibrat}  via a careful analysis
of the regular fibration $\phi_d:\pdgbar\to \mgbar$.

In the final Section \ref{bir-Pdg}, inspired by Lemma 8.1 in \cite{Cap}, we investigate for which values of $d$ and $d'$ the varieties  $P_{d,g}$ and $P_{d',g}$ are birational. We prove the following

\begin{Theorem}\label{bira-thm}
Assume that $g\geq 22$ or $g\geq 12$ and $(d-g+1,2g-2)=1$. Then $P_{d,g}$ is birational to $P_{d',g}$ if and only if $d'\equiv \pm d \mod (2g-2)$.
In this case, $P_{d,g}$ is isomorphic to $P_{d',g}$.
\end{Theorem}
This follows from Theorem \ref{birationa}, where we also determine the possible birational maps between the varieties
$P_{d,g}$ for $g$ big enough. From the same result, we obtain a description the group of birational self-maps of $P_{d,g}$
(see Corollary \ref{Cor-bira1}) and we deduce that the boundary of $\ov{P}_{d,g}$ is preserved by any automorphism of $\ov{P}_{d,g}$
(see Corollary \ref{Cor-bira2}).

While this work was being written down, Farkas and Verra posted on the arXiv the preprint \cite{FarVer},
where they determine, among other things, the Kodaira dimension of $P_{g,g}$ (note that the degree $g$ satisfies the assumptions of
our Theorem \ref{Kod-geom}, so that their result is a particular case of our main theorem).
However, their strategy is different from ours and it seems to apply only in the special case $d=g$.
Indeed, the authors of loc. cit. consider the global Abel-Jacobi map
$$A_{g,d}: M_{g,d}/S_d\to P_{d,g},$$
obtained by sending a curve $C$ together with a collection of unordered points $\{p_1,\ldots, p_d\}$ into the pair
$(C,\O_C(p_1+\cdots+p_d))$. It is well-known that the map $A_{g,d}$ is a birational isomorphism in degree $d=g$ (and only in this case).
Using this fact,  Farkas and Verra determine the Kodaira dimension of  $P_{g,g}$ by studying the pluricanonical forms on the Deligne-Mumford-Knudsen compactification
$\ov{M}_{g,g}/S_g$ (instead of the Caporaso compactification $\ov{P}_{g,g}$, as we do in this paper).

Throughout this paper, we work over the complex field $\mathbb{C}$. Moreover,  we fix two integers $g\geq 2$ and $d$.
%Note however that this assumption on the base field is used
%only in the proof of Theorem \ref{sing-Pdst} (because the Reid-Tai criterion is known only over $\C$) and Theorem \ref{birationa}.
%The other results of the paper hold over an arbitrary algebraically closed field.

\section{Preliminaries}

%\subsection{The stack $\pdb$ and the scheme $\pdgbar$}
\begin{nota}{\emph{The stack $\pdb$ and the scheme $\pdgbar$}}\label{def-stack-scheme}

In this subsection, we recall the definition of the stack $\Pdgbar$ and its good moduli space $\pdgbar$,
and collect some of their properties to be used later on.

Let $\Pdg$ be the universal Picard stack over the moduli stack $\Mg$ of smooth
curves of genus $g$. The fiber $\Pdg(S)$ of $\Pdg$ over a scheme $S$ is the groupoid whose objects are
families of smooth curves $\CC\to S$ endowed with a line bundle $\LL$ over $\CC$ of relative degree $d$ over $S$
and whose arrows are the obvious isomorphisms.
$\gd$ is a smooth irreducible (Artin)
algebraic stack of dimension $4g-4$ endowed with a natural forgetful map $\Phi_d: \Pdg \to \Mg$. The stack
$\Pdg$ admits a good moduli scheme $\pdg$ of dimension $4g-3$ which has a natural forgetful map $\phi_d: \pdg \to \mg$
onto the coarse moduli scheme of smooth curves of genus $g$.
We have the following commutative diagram

\begin{equation}\label{diag0}
\xymatrix{
\Pdg\ar[r] \ar@{->>}_{\Phi_d}[d] &\pdg\ar@{->>}^{\phi_d}[d]\\
\Mg \ar[r] & \mg.
}
\end{equation}

\begin{Warning}\label{gerbe}
The fact that $\Pdg$ has dimension $4g-4$ (and not $4g-3$ as $\pdg$)  is due to the fact that on each object $(\CC\to S,\LL)$ of $\Pdg(S)$
there is an action of the multiplicative group $\Gm$ via scalar multiplication on $\LL$. Therefore the map $\Phi_d$ factors as
$$\Phi_d:\Pdg\to \Pdg^{\Gm}\to P_{d,g},$$
where $\Pdg^{\Gm}$ (which is denoted by $\Pdg\fatslash \Gm$ by some authors) is the $\Gm$-rigidification of $\Pdg$ along the subgroup $\Gm$. Note that
$\Pdg^{\Gm}$ is a Deligne-Mumford stack of dimension $4g-3$ while $\Pdg$ is an Artin stack  of dimension
$4g-4$ which is {\rm not} Deligne-Mumford.  However, we will never need the rigidified stack $\Pdg^{\Gm}$ in this work
so that we refer to \cite[Sec. 4]{Melo} for more details
(note that in loc. cit. our stack $\Pdg$ is denoted by $\mathcal{G}_{d,g}$ while its rigidification $\Pdg^{\Gm}$ is denoted
by $\mathcal{P}_{d,g}$).
\end{Warning}

The stack $\Pdg$ and the scheme $\pdg$ have been compactified in a modular way in \cite{Cap}, \cite{Capbis} and \cite{Melo}.
To describe these compactifications, we need to recall some definitions.

\begin{Definition}\label{quasi-stable}
A connected, projective nodal curve $C$ is said to be \emph{quasistable} if it is (Deligne-Mumford) semistable and
the exceptional components of $C$ do not meet.
\end{Definition}
Given a quasi-stable curve $C$, we will denote by $C_{\rm exc}$ the subcurve of $C$ (called exceptional subcurve) given by the union of all the exceptional
components of $C$; by $\w{C}:=\overline{C\setminus C_{\rm exc}}$ its complementary subcurve (called non-exceptional
subcurve) and by $C^{\rm st}$ the stabilization of $C$. Moreover, we will denote by $\gamma(\w{C})$
the number of connected components of $\w{C}$.
%Adopting the terminology of \cite[Section 4.1]{Cap-Lis}, we give the following

\begin{Definition} \label{balanced}
Let $C$ be a quasistable curve of genus $g\ge 2$ and $L$ a degree $d$ line bundle on $C$.
\begin{enumerate}[(i)]
\item \label{bala1} We say that $L$
%(or its multidegree)
is \emph{balanced} if
 \begin{itemize}
 \item
 for every subcurve $Z$ of $C$ the following (``Basic Inequality'') holds
\begin{equation}\label{basic}
\frac{d \deg_Z ({\omega_C}_{|Z})}{2g-2}-\frac {k_Z}2\le \deg_ZL\leq \frac{d \deg_Z ({\omega_C}_{|Z})}{2g-2}+\frac{k_Z}2,
\end{equation}
where $k_Z$ is the number of intersection points of $Z$ with the complementary subcurve $Z^c:=\ov{C\setminus Z}$.
\item  $\deg_EL=1$ for every exceptional component $E$ of $C$.
\end{itemize}
\item \label{bala2}We say that $L$
%(or its multidegree)
is \emph{strictly balanced} if it is balanced and if for each proper subcurve
$Z$ of $C$ for which one of the two inequalities in (\ref{basic}) is not strict, then
the intersection $Z\cap Z^c$ is contained $C_{\rm exc}$.
\item \label{bala3} We say that $L$
%(or its multidegree)
is \emph{stably balanced} if it is balanced and if for each proper subcurve
$Z$ of $C$ for which one of the two inequalities in (\ref{basic}) is not strict, then
either $Z$ or $Z^c$ is entirely contained in $C_{\rm exc}$.
\end{enumerate}
\end{Definition}

%\begin{Remark}
The above Definitions \ref{balanced}\eqref{bala1} and \ref{balanced}\eqref{bala3} are taken from \cite[Def. 5.1.1]{CCC} (see also \cite[Def. 4.6]{Capbis})
and they are equivalent, respectively, to the definitions of semistable in \cite[Sec. 5.5]{Cap} and G-stable in \cite[Sec. 6.2]{Cap}.
The Definition \ref{balanced}\eqref{bala2} is taken from \cite[Sec. 4.1]{Cap-Lis} and it is equivalent to
the definition of {\rm extremal} in \cite[Sec. 5.2]{Cap}.
%\end{Remark}

There is an equivalence relation of the set of balanced line bundles on a quasi-stable curve $C$.

\begin{Definition}\label{equiv-rel}
Given two balanced line bundles $L$ and $L'$ on a quasi-stable curve $C$, we say that $L$ and $L'$ are {\rm equivalent}, and we write
$(C,L)\equiv (C,L')$, if $L_{|\w{C}}\cong L'_{|\w{C}}$. The equivalence class of a pair $(C,L)$ is denoted by $[(C,L)]$.
\end{Definition}

Note that the above equivalence relation $\equiv$ clearly preserves the multidegree of the line bundles, hence it preserves the condition of being strictly balanced or
stably balanced.

\begin{Remark}\label{GIT-inter}
In the GIT construction of $\pdgbar$ given in \cite{Cap}, the equivalence classes $[(C,L)]$ such that $C$ is quasi-stable and $L$ is
balanced (resp. strictly balanced, resp. stably balanced) correspond to the GIT-semistable (resp. GIT-polystable, resp.
GIT-stable) orbits (see \cite[Prop. 6.1, Lemma 6.1]{Cap} and also \cite[Thm. 5.1.6]{CCC}).
%Therefore $\pdst$ is the open subset of $\pdgbar$ where the GIT quotient is geometric.
\end{Remark}

The relationship between stably balanced and strictly balanced line bundles is given by the following

\begin{Lemma}\label{compa-bal}
A line bundle $L$ on a quasi-stable curve $C$ is stably balanced  if and only if it is strictly balanced
and $\w{C}$ is connected.
\end{Lemma}
\begin{proof}
Assume first that $L$ is strictly balanced and that $\w{C}$ is connected. Let $Z$ be a proper subcurve of $C$ such that
one of the two inequalities in \eqref{basic} is not strict. Then $Z\cap Z^c\subset C_{\rm exc}$ because $L$ is
strictly balanced by hypothesis. Therefore the non-exceptional subcurve $\w{C}$ can be written as a disjoint union of the two subcurves
$Z\cap \w{C}$ and $Z^c\cap \w{C}$. Since $\w{C}$ is connected by hypothesis, we must have that either $Z\cap \w{C}=\emptyset$
or $Z^c\cap \w{C}=\emptyset$, which implies that either $Z\subseteq C_{\rm exc}$ or $Z^c\subseteq C_{\rm exc}$, respectively.
 This shows that $L$ is stably balanced.

Conversely, assume that $L$ is stably balanced. Clearly this implies that $L$ is strictly balanced. Assume, by contradiction, that
$\w{C}$ is not connected. Then we can find two proper disjoint subcurves $D_1$ and  $D_2$ of $C$ that are not contained in
$C_{\rm exc}$ and such that  $E:=(D_1\cup D_2)^c$ is the union of $r\geq 1$ exceptional components of $C$.
It is easily checked that
\begin{equation*}
\left\{\begin{aligned}
 & \deg_{D_1\cup E}(\omega_C)=\deg_{D_1}(\omega_C), \\
 & k_{D_1\cup E}=k_{D_1}=r, \\
 & \deg_{D_1\cup E} L=\deg_{D_1} L +r.
\end{aligned}\tag{*}
\right.
\end{equation*}
Applying the inequality \eqref{basic} to the subcurves $D_1$ and $D_1\cup E$, we get
\begin{equation*}
%\left\{\begin{aligned}&
-\frac{r}{2}=-\frac{k_{D_1}}{2}\leq \deg_{D_1}L-d\frac{\deg_{D_1}(\omega_C)}{2g-1}=
\end{equation*}
\begin{equation*}
=\deg_{D_1\cup E} L-r-d\frac{\deg_{D_1\cup E}(\omega_C)}{2g-2}\leq \frac{k_{D_1\cup E}}{2}-r=-\frac{r}{2}.
%\\\end{aligned}\right.\tag{**}
\end{equation*}
Therefore one of the inequalities \eqref{basic} is strict fo the subcurve $D_1$ and this contradicts the fact that
$L$ is strictly balanced since $\emptyset \not\neq D_1\not\subseteq C_{\rm exc}$ by construction.
\end{proof}
%It is well known (and easy to prove) that $L$ is stably balanced on $C$ if and only if it is strictly balanced
%and $\w{C}$ is connected.

Let $\pdb$ be the category whose objects are
families of quasistable curves $\CC\to S$ endowed with a line bundle $\LL$ of relative degree $d$
whose restriction to each geometric fiber is balanced and whose arrows are Cartesian diagrams of such families.
Cleary $\pdb$ is a category fibered in groupoids over the category of schemes.
%fibered in groupoids whose fiber over a scheme $S$ is the groupoid whose objects are
%families of quasistable curves $\CC\to S$ endowed with a line bundle $\LL$ of relative degree $d$
%whose restriction to each geometric fiber is balanced and whose arrows are the obvious isomorphisms.
The following theorem summarizes some of the properties of $\Pdgbar$ and of its good moduli space $\pdgbar$
known thanks to Caporaso and Melo (note that our stacks $\Pdg$ and $\Pdgbar$ are called $\mathcal{G}_{d,g}$
and $\overline{\mathcal G}_{d,g}$ in \cite{Melo}).

\begin{Theorem} \emph{(\cite{Cap}, \cite{Capbis}, \cite{Melo})}
\noindent
\begin{enumerate}
\item $\Pdgbar$ is an irreducible, smooth and universally closed Artin stack of finite type over $\C$ and of dimension $4g-4$.
It contains the stack $\Pdg$ as a dense open substack.
\item $\Pdgbar$ admits a good moduli space $\pdgbar$, that is a normal irreducible projective variety of dimension $4g-3$.
The geometric points of $\pdgbar$ correspond bijectively to the
equivalence classes of pairs $(C,L)$ where $C$ is a quasi-stable  curve of genus $g$ and $L$ is a strictly balanced
line bundle of degree $d$.
\item $\ov{P}_{d,g}$ is a coarse moduli scheme for $\pdb$ if and only if $(d+1-g, 2g-2)=1$.
In this case $\pdgbar$ has only finite quotient singularities.
%\item The natural ``forgetful'' morphism $\Phi_d:\pdb\to \mgb$ is of finite type, universally closed and surjective. Moreover we have the following
 %Cartesian diagram:
%$$\xymatrix{
%\pd\ar@{^{(}->}[r] \ar@{->>}[d] &\pdb \ar@{->>}^{\Phi_d}[d]\\
%\mg \ar@{^{(}->}[r] & \mgb.
%}$$
%\item The following conditions are equivalent:
%\begin{enumerate}
%\item $(d+1-g, 2g-2)=1$;
%\item $\Phi_d$ is separated;
%\item $\Phi_d$ is (strongly) representable;
%\item $\pdb$ is separated;
%\item $\pdb$ is a Deligne-Mumford stack.
%\end{enumerate}
\end{enumerate}
\end{Theorem}

The construction of the scheme $\pdgbar$ as a GIT-quotient is due to Caporaso (see \cite{Cap});
the construction of the stack $\Pdgbar$ is due to Caporaso (see \cite{Capbis}) in the case $(d+1-g,2g-2)=1$
and to Melo (see \cite{Melo}) in the general case.
Note that we have a natural commutative diagram compactifying the diagram \eqref{diag0}:
%forgetful maps
\begin{equation}\label{diag-maps}
\xymatrix{
\Pdgbar\ar[r] \ar@{->>}_{\Phi_d}[d] &\pdgbar \ar@{->>}^{\phi_d}[d]\\
\Mgbar \ar[r] & \mgbar.
}
\end{equation}

\begin{Notation}\label{conv-points}
>From now on, for the ease of notation, whenever we write $(C, L)\in \pdgbar$ we mean that $L$ is a strictly balanced line bundle on the quasi-stable curve $C$,
considered up to the equivalence relation of Definition \ref{equiv-rel}.
\end{Notation}

Next we introduce an open subset of $\pdgbar$ that will play a special role in the sequel.

\begin{Definition}
We denote by  $\pdst$ the open subset of  $\pdgbar$
consisting of pairs $(C,L)\in \pdgbar$ where  $L$ is stably balanced.
\end{Definition}

By the above Remark \ref{GIT-inter},  $\pdst$ is the open subset of $\pdgbar$ where the GIT quotient is geometric.
In \cite[Lemma 2.2]{Cap}, it is proved that the semistable locus (called $H_d$ in loc. cit.)  inside the Hilbert scheme whose GIT quotient gives
$\pdgbar$ is smooth. From this, it follows that
$\pdst$ has finite quotient singularities  (see (\ref{localring2})
for an explicit local description). Moreover,   $\pdst=\pdgbar$ if and only if $(d+1-g,2g-2)=1$ by
 \cite[Prop. 6.2]{Cap}.

Albeit $\ov{P}_{d,g}$ has not necessarily finite quotient singularities, we have
the following useful result (see the proof of \cite[Cor. 1]{Fon}):
\begin{Fact}\emph{(\cite{Fon})}
$\ov{P}_{d,g}$ is a $\Q$-factorial variety.
\end{Fact}
In view of the above result, we will identify throughout this paper $\Q$-Weil divisors and $\Q$-Cartier divisors on $\ov{P}_{d,g}$.

\end{nota}

%\subsection{The automorphism group $\Aut(C,L)$}

\begin{nota}{\emph{The automorphism group $\Aut(C,L)$}}\label{Sec-auto}

%We will start by describing the automorphism group of a pair $(C,L)\in \pdgbar$.
%Recall that $C$ is a quasi-stable curve and $L$ is a GIT-polystable line bundle on $C$
%(this is equivalent to being extremal in the sense of \cite[Sec. 5.2]{Cap} or
%strictly balanced in the sense of \cite[Def. 4.1.4(2)]{Cap-Lis}).
For later use, we  describe the automorphism group of a pair $(C,L)$ consisting of a quasi-stable curve $C$ and 
a balanced line bundle $L$ on $C$.
An automorphism of $(C, L)$ is given by a pair $(\sigma, \psi)$ such that $\sigma\in \Aut(C)$ and
$\psi$ is an isomorphism between the line bundles $L$ and $\sigma^*(L)$. The group of automorphisms of $(C,L)$  is denoted
by $\Aut(C,L)$. We get a natural forgetful homomorphism
\begin{equation}\label{homo-auto}
\begin{aligned}
F: \Aut(C,L)&\to \Aut(C)\\
(\sigma,\psi)& \mapsto \sigma\\
\end{aligned}
\end{equation}
whose kernel is the  multiplicative group $\Gm$,  acting
 as fiberwise multiplication on $L$, and whose image is the subgroup of $\sigma\in \Aut(C)$ such that $\sigma^*(L)\cong L$.
%The automorphisms belonging to $\Gm\subseteq \Aut(C,L)$  are called  {\it scalar} automorphisms.
The quotient $\Aut(C,L)/\Gm$ is denoted by $\ov{\Aut(C,L)}$
and is called the {\it reduced} automorphism group of $(C, L)$.
Note that $\Aut(C,L)$ depends only on the equivalence class $[(L,C)]$ (see Def. \ref{equiv-rel}).

By composing the above homomorphism $F$ of \eqref{homo-auto} with the natural homomorphism $\Aut(C)\to \Aut(C^{\rm st})$ induced
by the stabilization map $C\to C^{\rm st}$, we get a homomorphism
$$G: \Aut(C,L)\to \Aut(C^{\rm st}),$$
whose kernel is described in the next Lemma.

\begin{Lemma}\label{automo}
We have a commutative diagram with exact rows
\begin{equation*}
\xymatrix{
0 \ar[r] & \Gm^{\gamma(\w{C})} \ar[r]\ar@{->>}[d]&  \Aut(C,L) \ar[r]^{G}\ar@{->>}[d]&  \Aut(C^{\rm st})\ar@{=}[d]\\
0 \ar[r] & \Gm^{\gamma(\w{C})}/\Gm \ar[r] &  \ov{\Aut(C,L)} \ar[r]^{\ov{G}} &  \Aut(C^{\rm st}),\\
}
\end{equation*}
where $\Gm\subseteq \Gm^{\gamma(\w{C})}$ is the diagonal embedding.
\end{Lemma}
\begin{proof}
The exactness of the first row is proved using an argument similar to the one used in  the proof of \cite[Lemma 2.3.2]{CCC}.
We sketch the argument for the sake of completeness.
Let $E_1,\ldots,E_m$ be the exceptional components of $C$ and let $X_1,\ldots,X_{\gamma(\w{C})}$ be the connected components of $\w{C}$.
We identify each $E_i$ to a copy of  $\P^1$ attached to the rest of the curve at the points $0$ and $\infty$.
An element $(\sigma, \psi)\in \Aut(C,L)$ belongs to the kernel of the
map $\Aut(C,L)\to \Aut(C^{\rm st})$ if and only if  $\sigma_{|\w{C}}=\id_{\w{C}}$  and $\sigma$ acts as multiplication by $m_i\in \Gm(k)=k^*$
on the exceptional component $E_i$. If we restrict the isomorphism $\psi:L\stackrel{\cong}{\to} \sigma^*(L)$ to $\w{C}$, we get that $\psi$ is the fiberwise multiplication
by $l_j\in \Gm(k)=k^*$ on each line bundle $L_{|X_j}$. The scalars $m_i$ are uniquely determined by the scalars $l_j$: if $0\in E_i$ lies on the component
$X_j$ and $\infty \in E_i$ lies on the component $X_h$ (possibly with $j=h$), then by the compatibility between $\sigma$ and $\psi$ we get that $m_i=l_j/l_h$
(see the proof of \cite[Lemma 2.3.2]{CCC}).
Therefore the element $(\sigma,\psi)$ is uniquely determined by the scalars $l_1,\ldots,l_{\gamma(\w{C})}$ and we are done.

From the above proof, it is clear that the homomorphisms corresponding
to the diagonal embedding $\Gm\hookrightarrow \Gm^{\gamma(\w{C})}$ are exactly the fiberwise automorphisms on $L$, hence the exactness of the second row follows.
%We simply observe that the homomorphisms $\Gm^{\gamma(\w{C})} \hookrightarrow  \Aut(C,L)$ induce on $L$ the multiplication
%by scalars on each connected component of $\widetilde{C}$. Therefore it is clear that the homomorphisms corresponding
%to the diagonal embedding $\Gm\hookrightarrow \Gm^{\gamma(\w{C})}$ are exactly the scalar automorphisms,
%hence we are done.
\end{proof}

\begin{Corollary}\label{Cor-auto}
%\noindent
%\begin{enumerate}[(i)]
%\item \label{Cor-aut1} Let $C$ be a quasi-stable curve of genus $g$ and a line bundle $L$ on $C$ of degree $d$.
%Then $(C,L)$ is stably balanced if and only if $(C,L)$ is strictly balanced and $\w{C}$ is connected.
%\item\label{Cor-aut2}
If $(C,L)$ is stably balanced then $\ov{\Aut(C,L)}$ is a subgroup of  $\Aut(C^{\rm st})$. In particular, $\ov{\Aut(C, L)}$ is a finite group.
%\end{enumerate}
\end{Corollary}
\begin{proof}
The first assertion follows from the last row of the diagram in Lemma \ref{automo} together with the fact that if $L$ is stably balanced on $C$ then
$\w{C}$ must be connected by Lemma \ref{compa-bal}. The last assertion follows from the first one together with the well-known fact that
the automorphism group of a stable curve is finite.
\end{proof}

\end{nota}

\begin{nota}{\emph{The local structure of $\pdgbar$}}\label{loc-struc}

The complete local ring
$\widehat{\O}_{\pdgbar,(C,L)}$ of $\pdgbar$ at a point $(C,L)$ can be described
using the deformation theory of pairs $(C,L)$, which can be found in \cite[Sec. 3.3.3]{ser}
for $C$ smooth and has been extended to the singular case in \cite{wang}.

% acts on $\Def_{(C,L)}$ naturally. $\Def_{(C,L)}/\Aut(C,L)$ is the quotient by this action.

%The semiuniversal deformation space can be described by using the results of J. Wang (\cite{wang})
%on the deformation functor of the pair $(C,L)$.
Let us recall the results of  \cite{wang}.
We denote by $\PP^1_C(L)$ the sheaf of one jets (or sheaf of principle parts) of
$L$ on $C$ (\cite[Sec. 2]{wang}). The sheaf $\PP^1_C(L)$ fits into
an exact sequence (see \cite[Eq. (2.1)]{wang})
\begin{equation}\label{jets}
0 \to \Omega^1_C \to \PP^1_C(L)\otimes L^{-1} \to \O_C \to 0.
\end{equation}
Explicitly, the above extension \eqref{jets} can be described as follows (see \cite[p. 145]{ser}).
Let $\O_C^*\rightarrow \Omega^1_C$ be the homomorphism of sheaves
given by sending
$u\in \Gamma(U, {\mathcal O}_C^*)$  into $\displaystyle \frac{d u}{u}\in \Gamma(U,\Omega_C^1)$
for any open subset $U\subseteq C$. By passing to cohomology, we get a group homomorphism
$\theta_C:\Pic(C)=H^1(C,\O_C^*) \rightarrow H^1(C, \Omega_C^1)$.
By using the identification $H^1(C, \Omega_C^1) \cong \Ext^1({\mathcal O}_C, \Omega_C^1)$,
the map $\theta_C$ sends the line bundle $L$ to the class of the extension \eqref{jets}.
%\begin{equation}\label{jets-bis}
%0 \rightarrow \Omega_C^1 \rightarrow{\mathcal E}_L \rightarrow {\mathcal O}_C\rightarrow 0,
%\end{equation}
%which is the extension \eqref{jets} tensored with $L^{-1}$.

Wang in \cite{wang} proves that the sheaf $\PP^1_C(L)$ controls the tangent and obstruction theory of the
pair $(C,L)$. Let us denote by $\Def_{(C,L)}$ the functor of infinitesimal deformations of the pair $(C,L)$ (see \cite[p. 146]{ser}) and by
$T_{\Def_{(C,L)}}$ the tangent space to $\Def_{(C,L)}$ (in the sense of \cite[Lemma 2.2.1]{ser}).

\begin{Theorem}\emph{(\cite{wang})} \label{wang-thm}
\noindent
\begin{enumerate}[(i)]
\item \label{wang-thm1} We have that $T_{\Def_{C,L)}}=\Ext^1(\PP^1_C(L),L)$.
%The tangent space to $\Def_{(C,L)}$ is equal to
%$T_{\Def_{C,L)}}=\Ext^1(\PP^1_C(L),L)$.
\item \label{wang-thm2} An obstruction space for $\Def_{(C,L)}$   is given by $\Ext^2(\PP^1_C(L),L)$.
\end{enumerate}
\end{Theorem}
\begin{proof}
Part (i) is \cite[Thm. 3.1(1)]{wang}; Part (ii) is \cite[Thm. 4.6(a)]{wang}.
\end{proof}

Moreover, the infinitesimal automorphisms of the pair $(C,L)$ are governed by
$\Ext^0(\PP_C^1(L),L)$, as shown by the following Lemma.

\begin{Lemma}\label{inf-aut}
The tangent space of $\Aut(C,L)$ at the identity is equal to $\Ext^0(\PP^1_C(L),L)$.
\end{Lemma}
\begin{proof}
The Lemma is certainly well-known to the experts, at least in the case where $C$ is smooth. However, we include
a proof for the lack of a suitable reference.

According to the discussion in \ref{Sec-auto}, we have an exact sequence of groups
\begin{equation*}\label{seq-aut}
0 \to \Gm \to \Aut(C,L)\to {\rm Stab}_L(\Aut(C))\to 0,
\end{equation*}
where ${\rm Stab}_L(\Aut(C))$ is the stabilizer of $L$ in $\Aut(C)$, i.e. the subgroup of $\Aut(C)$ consisting of all the
automorphisms $\sigma$ of $C$ such that $\sigma^*(L)\cong L$. In other words, ${\rm Stab}_L(\Aut(C))$ is the image of
$\Aut(C,L)$ via the map $F$ of \eqref{homo-auto}.
By passing to the tangent spaces at the origin, we get
\begin{equation}\label{sequ1}
0 \to T_0\Gm=k \to T_0\Aut(C,L)\to T_0{\rm Stab}_L(\Aut(C))\to 0,
\end{equation}
where we have denoted by $0$ the identity element in each of the above groups.

On the other hand, by dualizing \eqref{jets}, we get the short exact sequence
\begin{equation}\label{dual-seq}
0\to \O_C \to \PP^1(L)^{\vee}\otimes L\to T_C\to 0.
\end{equation}
Passing to cohomology, we get the  exact sequence
\begin{equation}\label{sequ2}
%\begin{aligned}
0\to H^0(C, \O_C)=k \to  H^0(C, \PP^1(L)^{\vee}\otimes L)\stackrel{p}{\to} H^0(C, T_C)
%& \to H^0(C, T_C)\stackrel{\partial}{\to} H^1(C,\O_C). \\
%\end{aligned}
\end{equation}
Compare now the exact sequences \eqref{sequ1} and \eqref{sequ2}. Since
$T_0\Aut(C)=H^0(C, T_C)$ by \cite[Prop. 2.6.2]{ser} and clearly
$\Ext^0(\PP^1(L),L)=H^0(C,$ $ \PP^1(L)^{\vee}\otimes L)$, it is enough to show that
\begin{equation}\label{sequ3}
T_0{\rm Stab}_L(\Aut(C))={\rm Im}(p).
\end{equation}
Let ${\mathcal U}=\{U_{\alpha}\}$ be an affine open covering of $C$ trivializing $L$ and let
$f_{\alpha \beta} \in \Gamma(U_{\alpha \beta}, {\mathcal O}_C^*)$ the transition functions of $L$ with
respect to ${\mathcal U}$, where as usual $U_{\alpha\beta}:=U_{\alpha}\cap U_{\beta}$.
%such that $L$ is represented by a system of transition functions $\{f_{\alpha\beta}\}$ with $f_{\alpha \beta} \in %\Gamma(U_{\alpha \beta}, {\mathcal O}_C^*)$.
Then $\theta_C(L)\in H^1(C,\Omega_C^1)$ is represented by the \v{C}ech $1$-cocycle $\left(\frac{df_{\alpha \beta}}{f_{\alpha \beta}}\right)\in Z^1({\mathcal U}, \Omega_C^1)$ (see \cite[p. 145]{ser}).
>From the exact sequence \eqref{dual-seq}, it follows that the sheaf $(\PP^1(L)^{\vee}\otimes L)_{|U_{\alpha}}$ is isomorphic to
$(\O_C)_{U_{\alpha}}\oplus (T_C)_{|U_{\alpha}}$
and an element of $\Ext^0(\PP^1(L), L)=H^0(C, \PP^1(L)^{\vee}\otimes L)$ is represented by a \v{C}ech $0$-cochain $$(k_{\alpha},d_{\alpha}) \in C^0({\mathcal U}, \PP^1(L)^{\vee}\otimes L)=C^0({\mathcal U}, \O_C)\oplus
C^0({\mathcal U}, T_C),$$
which satisfies the cocycle conditions:
$d_{\alpha}=d_{\beta}$ and $k_{\beta}-k_{\alpha}=\frac{d_{\alpha}(f_{\alpha\beta})}{f_{\alpha\beta}}$ on
$U_{\alpha\beta}=U_{\alpha}\cap U_{\beta}$ (see \cite[p. 145]{ser}). Since the $f_{\alpha\beta}$'s are the transition
functions of $L$, we conclude that the image of $p$ consists of all the \v{C}ech $0$-cocycles
$(d_{\alpha})\in Z^0({\mathcal U}, T_C)$ corresponding to the infinitesimal automorphisms of
$C$ which preserve the line bundle $L$. In other words, \eqref{sequ3} is satisfied and we are done.

\end{proof}

We can now compute the dimension of the vector spaces $\Ext^i(\PP^1_C(L),L)$.

\begin{Lemma}\label{ext-groups}
We have that
$$\left\{\begin{aligned}
&\dim \Ext^i(\PP^1_C(L),L)=0  \text{ for } i\geq 2, \\
&\dim \Ext^1(\PP^1_C(L), L)=4g-4+\gamma(\w{C}), \\
&\dim \Ext^0(\PP^1_C(L),L)=\gamma(\w{C}).
\end{aligned}\right.
$$
\end{Lemma}
\begin{proof}
By applying the functor $\Hom(-,\O_C)$ to the exact sequence \eqref{jets} and using that
$\Ext^{\geq 2}(\O_C,\O_C)=H^{\geq 2}(C,\O_C)=0$ since $C$ is a curve and that $\Ext^{\geq 2}(\Omega_C^1, \O_C)=0$
by \cite[Lemma 1.3]{DM}, we get the vanishing $\Ext^{\geq 2}(\PP^1_C(L),L)=0$.
The fact that $\Ext^0(\PP_C^1(L),L)=\gamma(\w{C})$ follows from Lemma \ref{automo} and \ref{inf-aut}.
Finally,  from the exact sequence (\ref{jets}), we get
$$\dim \Ext^0(\PP^1_C(L),L)- \dim \Ext^1(\PP^1_C(L),L)=
$$
$$
\dim \Ext^0(\O_C,\O_C)-\dim \Ext^1(\O_C,\O_C)+$$
$$\dim \Ext^0(\Omega_C^1,\O_C)-\dim \Ext^1(\Omega_C^1,\O_C)=$$
$$-(g-1)-(3g-3)=-(4g-4),$$
from which we conclude.
\end{proof}

We can now prove that the functor  $\Def_{(C,L)}$ has a  semiuniversal formal element (in the sense of \cite[Def. 2.2.6]{ser}).

\begin{Proposition}\label{semiuniv}
\noindent
\begin{enumerate}[(i)]
\item \label{semiuniv1}�The functor $\Def_{(C,L)}$ has a semiuniversal formal element $\un{\Def}_{(C,L)}$.
\item \label{semiuniv2}�$\un{\Def}_{(C,L)}$ is equal to the formal spectrum of $k[[x_1,\ldots, x_{4g-4+\gamma(\w{C})}]]$.
%where $\Spf$ denotes the formal spectrum of a complete algebra.
\end{enumerate}
\end{Proposition}
\begin{proof}
Part \eqref{semiuniv1} is proved in \cite[Thm 3.3.11(i)]{ser} in the case where $C$ is smooth.
The proof of loc. cit. consists in showing that Schlessinger's conditions are satisfied and this
extends to the case where $C$ is nodal: the crucial point of the proof is
showing that $T_{\Def_{(C,L)}}$  is finite dimensional,
and this follows in our case from Theorem \ref{wang-thm}\eqref{wang-thm1}.

>From Theorem \ref{wang-thm} and Lemma \ref{ext-groups}, it follows that $\un{\Def}_{(C,L)}$ is formally smooth and
that the dimension of the tangent space at its unique closed point is $4g-4+\gamma(\w{C})$, from which part \eqref{semiuniv2}
follows.
\end{proof}

Now we can describe the complete local ring $\widehat{\O}_{\pdgbar,(C,L)}$ of $\pdgbar$ at a point $(C,L)$.
Note that the automorphism group $\Aut(C,L)$ acts on $\un{\Def}_{(C,L)}$ (hence on
$\C[[x_1,\ldots,x_{4g-4+\gamma(\w{C})}]]$) by the semiuniversality
of $\un{\Def}_{(C,L)}$.
%Denote by $\Spf(A)$ the formal spectrum of a complete local ring $A$.
By a standard argument based on Luna's \'etale slice theorem (see \cite[p. 97]{Lun} and also
\cite[Sec. II]{Las}, \cite[Sec. 7.4]{Dre}), the formal spectrum (which we denote by $\Spf$) of the complete local ring
$\widehat{\O}_{\pdgbar,(C,L)}$ of $\pdgbar$ at the point $(C,L)\in \pdgbar$
is given by
\begin{equation}\label{localring}
\Spf \widehat{\O}_{\pdgbar,(C,L)}=\un{\Def}_{(C,L)}/\Aut(C,L),
\end{equation}
where $\un{\Def}_{(C,L)}/\Aut(C,L)$ is the quotient of $\un{\Def}_{(C,L)}$ with respect to the natural action of
$\Aut(C,L)$. In other words,
$\un{\Def}_{(C,L)}/\Aut(C,L)$ is equal to the formal spectrum of the ring of invariants
$\C[[x_1,\ldots,$ $ x_{4g-4+\gamma(\w{C})}]]^{\Aut(C,L)}$ (see Proposition \ref{semiuniv}).

Clearly, the scalar automorphisms $\Gm\subseteq \Aut(C,L)$ act trivially on $\un{\Def}_{(C,L)}$
and thus we get  the alternative description:
\begin{equation}\label{localring2}
\Spf \widehat{\O}_{\pdgbar,(C,L)}=\un{\Def}_{(C,L)}/\ov{\Aut(C,L)}.
\end{equation}

Note that, from the above description and Lemma \ref{automo}, it follows that $\pdst$ is the open subset of $\pdgbar$
consisting of pairs $(C,L)\in \pdgbar$ such that $\w{C}$ is connected.

We can similarly describe the morphism $\phi_d:\pdgbar\to \mgbar$ locally at $(C,L)\in \pdgbar$.
Denote by $\un{\Def}_C$ (resp. $\un{\Def}_{C^{\rm st}}$) the semiuniversal formal element associated to the infinitesimal deformation functor
$\Def_C$ (resp. $\Def_{C^{\rm st}}$) of $C$ (resp. $C^{\rm st}$), see \cite[Cor. 2.4.2]{ser}.

Locally at $(C,L)\in \pdgbar$, the morphism $\phi_d:\pdgbar\to \mgbar$ is given by
\begin{equation}\label{mor-loc}
\Spf \widehat{\O}_{\pdgbar,(C,L)}=\un{\Def}_{(C,L)}/\ov{\Aut(C,L)}\to
\Spf \widehat{\O}_{\mgbar,C^{\rm st}}=\un{\Def}_{C^{\rm st}}/\Aut(C^{\rm st}),
\end{equation}
where the homomorphism of groups $\ov{\Aut(C,L)}\to \Aut(C^{\rm st})$ is the one
given by Lemma \ref{automo} and the morphism $\un{\Def}_{(C,L)}\to \un{\Def}_{C^{\rm st}}$
is the composition of the forgetful morphism $\un{\Def}_{(C,L)}\to \un{\Def}_C$ with the stabilization morphism $\un{\Def}_C\to \un{\Def}_{C^{\rm st}}$.
The induced morphism at the level of tangent spaces
\begin{equation}\label{mor-tang}
T_{\Def_{(C,L)}}=\Ext^1(\PP^1_C(L), L)\to T_{\Def_{C^{\rm st}}}=\Ext^1(\Omega^1_{C^{\rm st}}, \O_{C^{\rm st}})
\end{equation}
is given by composing the morphism $\Ext^1(\PP^1_C(L), L)\to \Ext^1(\Omega_{C}^1\otimes L, L)=\Ext^1(\Omega_{C}^1, \O_C)$
induced by the exact sequence (\ref{jets}) with the morphism $\Ext^1(\Omega_{C}^1, \O_C)\to
\Ext^1(\Omega_{C^{\rm st}}, \O_{C^{\rm st}})$ induced by the stabilization map $C\to C^{\rm st}$. More precisely,
let $f:C \rightarrow C^{\rm st}$ be the stabilization morphism and denote by $Lf^*$ (resp. $Rf_*$) the left
derived functor of $f^*$ (resp. the right derived functor of $f_*$).
We have a natural map
\begin{equation}
\label{ext}
\Ext^1(\Omega_{C}^1,{\mathcal O}_{C}) \to \Ext^1(Lf^*\Omega^1_{C^{\rm st}},{\mathcal O}_C)=
\Ext^1(\Omega^1_{C^{\rm st}},{\mathcal O}_{C^{\rm st}}).
\end{equation}
The first map in \eqref{ext} is induced by the composite map $Lf^*\Omega^1_{C^{\rm st}}\to L^0f^*\Omega^1_{C^{\rm st}}=
f^*\Omega^1_{C^{\rm st}}\to \Omega^1_C$ in the derived category of coherent sheaves on $C$.
The equality in \eqref{ext} follows from the adjointness of the functors $Lf^*$ and $Rf_*$ between the derived category of coherent sheaves on $C$ and on $C^{\rm st}$, together with the fact that $Rf_*{\mathcal O}_C\cong {\mathcal O}_{C^{\rm st}}$
because $f$ is a sequence of blow-ups with projective spaces as fibers.

%Notice that $Rf_*{\mathcal O}_C\cong {\mathcal O}_{C^{\rm st}}$ because $f$ is a sequence of blow-ups with projective spaces as %fibers. We thus have
%\begin{equation}
%\label{ext}
%\Ext^1(\Omega_{C}^1,{\mathcal O}_{C}) \to \Ext^1(Lf^*\Omega^1_{C^{\rm st}},{\mathcal O}_C)=
%\Ext^1(\Omega^1_{C^{\rm st}},{\mathcal O}_{C^{\rm st}}),
%\end{equation}
%where $Lf^*$ is the left derived functor. The isomorphism in \eqref{ext} is a consequence of the adjointness of the functors %$Lf^*$ and $Rf_*$ between the derived category of coherent sheaves on $C$ and on $C^{\rm st}$. Since there exists a morphism %$Lf^*\Omega^1_{C^{\rm st}} \rightarrow L^0f^*\Omega^1_{C^{\rm st}}=f^*\Omega^1_{C^{\rm st}}$, we get a homomorphism from %$\Ext^1(f^*\Omega^1_{C^{\rm st}}, {\mathcal O}_C)$ to $\Ext^1(\Omega^1_{C^{\rm st}},{\mathcal O}_{C^{\rm st}})$.  The morphism %$f^*\Omega^1_{C^{\rm st}} \rightarrow \Omega^1_C$ yields a morphism from $\Ext^1(\Omega_{C}^1,{\mathcal O}_{C})$ to %$\Ext^1(f^*\Omega^1_{C^{\rm st}}, {\mathcal O}_C)$, and so we have a homomorphism from $\Ext^1(\Omega_{C}^1,{\mathcal O}_{C})$ %to $\Ext^1(f^*\Omega^1_{C^{\rm st}}, {\mathcal O}_{C^{\rm st}})$.

\end{nota}

\section{The fibration $\phi_d:\pdgbar\to \mgbar$}\label{Fibrat}

The aim of this section is to prove Proposition \ref{kod-fib} below, which gives the
second part of Theorem \ref{kod-nongeo}.

To this aim, we analize the natural morphism $\phi_d: \pdgbar \to \mgbar$. Note
that $\phi_d$ is a regular fibration
(i.e. a proper, surjective morphism with connected fibers), whose general fiber
is the degree-$d$ Jacobian $\Pic^d(C)$ of a general $[C] \in \mgbar$. Following Kawamata (see
\cite[Sec. 1]{kaw} and \cite[Cor. 7.3]{kaw}), we define the variation ${\rm Var}(\phi_d)$  of $\phi_d$
to be
\begin{equation}\label{def-var}
{\rm Var}(\phi_d)=\dim \mgbar- \dim \Ker(\delta_C),
\end{equation}
where
\begin{equation}\label{Kod-Spe-map}
\delta_C:T_{\mgbar,C}=H^1(C,T_C)\to H^1(\Pic^d(C),T_{\Pic^d(C)})
\end{equation}
is the Kodaira-Spencer map associated to  $\phi_d$ at a general point $C\in \mgbar$.

\begin{Lemma}\label{maxvar}
The algebraic fiber space $\phi_d:\pdgbar\to \mgbar$ has maximal variation,
i.e. ${\rm Var}(\phi_d)=3g-3$.
\end{Lemma}

\begin{proof}
%According to \cite[Cor. 7.3]{kaw}, the variation of $\phi_d$ is given by
%${\rm Var}(\phi_d)=\dim \mgbar- \dim \Ker(\delta_C),$
%where $\delta_C:T_{\mgbar,C}=H^1(C,T_C)\to H^1(\Pic^d(C),T_{\Pic^d(C)})$
%is the Kodaira-Spencer map of the map $\phi_d$ at a general point $[C]\in \mgbar$
%(hence, in particular, $C$ is smooth).
By \eqref{def-var}, we have to prove  the injectivity of the Kodaira-Spencer map $\delta_C$ for a general curve $C\in \mgbar$.

We will reinterpret the above Kodaira-Spencer map as composition of certain maps
that were studied in \cite[Sec. 2]{OS}, in their analysis of the local Torelli problems
for curves. We need to recall their setting, with the simplification that, since
we are only interested in the general curve, we can work
directly with the coarse moduli spaces $M_g$ and  $A_g$ (where, as usual, $A_g$ denotes the coarse moduli space of principally polarized
abelian varieties of dimension $g$),
without having to pass
to their $n$-level covers.
Consider the following commutative diagram (see \cite[p. 169]{OS}):
$$\xymatrix@=.7pc{
T_{M_g,C}\ar^(.4){k_C}_(.4){\cong}[r] \ar_{d t_g}[dd]& H^1(C, T_C) \ar^{d u}[d]\\
& H^1(C, u^*T_{\Pic^d(C)}) \\
T_{A_g,\Pic^d(C)}\ar^{k_{\Pic^d(C)}}[dr]  & \\
& H^1(\Pic^d(C),T_{\Pic^d(C)})\ar^{\cong}_{u^*}[uu].
}$$
where $u:C\to \Pic^d(C)$ is an Abel-Jacobi map (well-defined only up to translation),
$t_g:M_g\to A_g$ is the classical Torelli map, $k_C$ is the Kodaira-Spencer map in $C$
associated to the universal family over an open subset of $M_g$ containing $C$
and $k_{\Pic^d(C)}$ is the Kodaira-Spencer map in $\Pic^d(C)$ associated to the
universal family over an open subset of $A_g$ containing $\Pic^d(C)$.
The map $k_C$ is an isomorphism (see e. g. \cite[Thm. 2.2]{OS})
and the map $u^*$ is an isomorphism since $T_{\Pic^d(C)}$ is the trivial bundle
of rank $g$ and
$$u^*: H^1(\Pic^d(C), \O_{\Pic^d(C)}) \stackrel{\cong}{\longrightarrow} H^1(C, \O_C).$$
It is easy to see that
$$\delta_C=k_{\Pic^d(C)}\circ d t_g=(u^*)^{-1}\circ du \circ k_C.$$
Therefore, the injectivity of $\delta_C$ is equivalent to the injectivity of $du$.
According to \cite[Thm. 2.6]{OS}, the map $du$ is the dual of the multiplication map
$$\mu: H^0(C, \omega_C)\otimes H^0(C, \omega_C)\to H^0(C, \omega_C^{\otimes 2}),$$
which is well-known (Noether's theorem) to be surjective if $g=2$ or if $g\geq 3$ and
$C$ is not hyperelliptic. Since we assumed $C$ to be generic, we deduce the
injectivity of $\delta_C$ and we are done.
\end{proof}

\begin{Proposition}\label{kod-fib}
\noindent
\begin{enumerate}
\item We have that $\kappa(\pdgbar) \le 3g-3$.
\item If $\kappa(\mgbar)\geq 0$, then $\kappa(\pdgbar)= 3g-3$ and
the map $\phi_d:\pdgbar\to \mgbar$ is the Iitaka fibration of $\pdgbar$.
\end{enumerate}
\end{Proposition}

\begin{proof}
The subaddivity of the Kodaira dimension  (see \cite[Thm. 6.12]{Ueno}) applied to the regular fibration $\phi_d$ gives that
$$\kappa(\pdgbar)\leq \dim \mgbar+\kappa(\phi_d^{-1}(C)),$$
for a general $C\in \mgbar$. Since, for a general $C\in \mgbar$,
the fiber $\phi_d^{-1}(C)=\Pic^d(C)$ is an abelian variety,
we have that $\kappa(\phi_d^{-1}(C))=0$, which proves part (1).

Assume now that $\kappa(\mgbar)\geq 0$.
Observe that $\Pic^d(C)$ is a good minimal model, since
it is smooth and the canonical $K_{\Pic^d(C)}$ is trivial and thus
clearly semi-ample. Therefore, the Iitaka conjecture (in the
stronger form of \cite[p. 1]{kaw}) does hold true by \cite[Cor. 1.2]{kaw}
and gives that
$$\kappa(\pdgbar)\geq \kappa(\phi_d^{-1}(C))+{\rm max}\{\kappa(\mgbar),{\rm Var}(\phi_d)\}=3g-3,
$$
using the above Lemma \ref{maxvar}.
This, combined with part (1), proves that $\kappa(\pdgbar)=3g-3$.

The last part follows from the birational characterization of the Iitaka fibration
(see e. g. \cite[Thm. 6.11]{Ueno}) since $\phi_d:\pdgbar\to \mgbar$ is an algebraic fiber space such
that $\dim \mgbar=\kappa(\pdgbar)$ and the generic fiber $\phi_d^{-1}(C)=\Pic^d(C)$
is smooth and irreducible of Kodaira dimension zero.
\end{proof}

\section{The singularities of $\pdst$ }\label{Sec-sing}

The purpose of this section is to study the singularities of $\pdgbar$ with the aim of proving Theorem \ref{conj-sing}.
More generally, we will prove a similar statement (see Theorem \ref{sing-Pdst}) for the open subvariety $\pdst\subseteq \pdgbar$
of (\ref{def-stack-scheme}) and for any degree $d$.
Since $\pdst=\pdgbar$ if (and only if) $(d+1-g,2g-2)=1$ (by \cite[Prop. 6.2]{Cap}),
Theorem \ref{conj-sing}  is a special case of
Theorem \ref{sing-Pdst} below.
The need of restricting ourselves to the open subset $\pdst$ is due to the fact that
$\pdst$ has finite quotient singularities
and therefore we can apply the Reid--Tai criterion for the canonicity of finite quotient singularities (see e.g. \cite[pp. 27-28]{HarMum}
or \cite[Thm. 4.1.11]{Lud}).
%(in the strong form given in \cite[Appendix 1 to \S1]{HarMum}) to lift pluricanonical forms from its smooth locus to a
%desingularization of it.
For simplicity, we assume throughout this section that $g\geq 4$ in order to avoid problems with the
hyperelliptic locus in $\mgbar$.

In the analysis of the singularities of $\pdst$, the pairs $(C,L)\in \pdst$ such that $C$ contains an elliptic tail will play a special role. Let us give some definitions.

\begin{Definition}\label{def-ell-tail}
A connected subcurve $E$ of a quasi-stable curve $C$ is called an \emph{elliptic tail} if it has
arithmetic genus $1$ and meets the rest of the curve in exactly one node $P$ which is called
an elliptic tail node.
%We will always assume that $P$ is the origin of the group structure on
%the regular locus of $E^{\rm reg}$.
\end{Definition}

The following remark is straightforward.

\begin{Remark}\label{rmk-ell-tail}
If a quasi-stable curve $C$ has an elliptic tail $E\subseteq C$ then
the image $E^{\rm st}:=st(E)\subseteq C^{\rm st}$ of $E$ via the stabilization morphism $st: C\to C^{\rm st}$
is an elliptic tail of $C^{\rm st}$. Conversely, if $C^{\rm st}$ has an elliptic tail $E'\subseteq C^{\rm st}$,
then $E:=st^{-1}(E')$
%the inverse image $E\subseteq C$ of $E'$ via the stabilization morphism $C\to C^{\rm st}$
is an elliptic tail of $C$ such that $E'=E^{\rm st}$.
\end{Remark}

In the next Lemma, we describe the pairs $(C,L)\in \pdst$ such that $C$ has an elliptic tail $E\subseteq C$.
\begin{figure}[h]
\begin{center}
%\mbox{
\epsfig{file=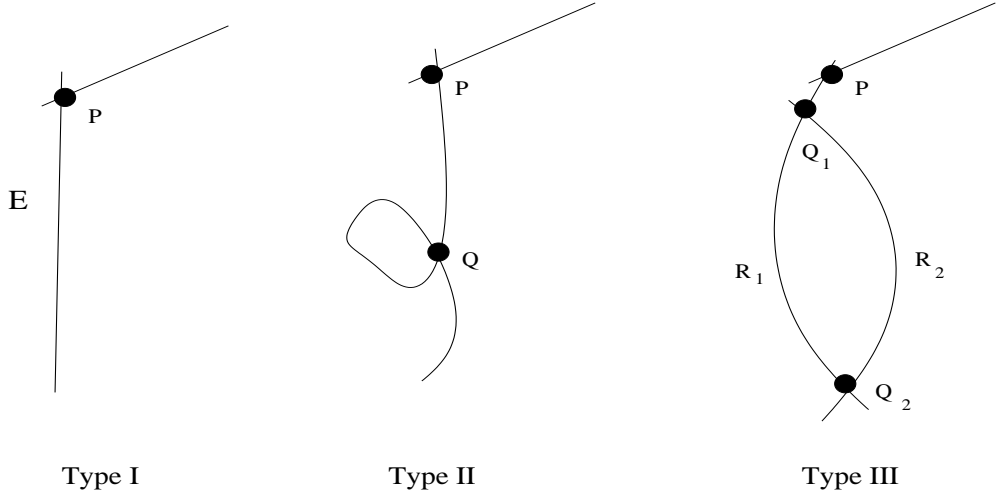,width=13cm,height=6.5cm,angle=0}
%}
\end{center}
\caption{The possible elliptic tails of a quasi-stable curve $C$ such that $(C,L)\in\pdst$ for some line bundle $L$.}
\label{Fig-tails}
\end{figure}

%\begin{Notation}\label{Ell-tail}

\begin{Lemma}\label{tail-stable}
Let $(C,L)\in \pdst$ such that $C$ has an elliptic tail $E\subseteq C$.
Then we have:
\begin{enumerate}[(i)]
\item \label{el-tail1} $d\not\equiv g-1 \mod (2g-2)$ and the degree
$\deg_E(L)$ of $L$ on $E$ is the unique integer $d_E$ such that
\begin{equation}\label{deg-tail}
-\frac{1}{2}< d_E-\frac{d}{2g-2}< \frac{1}{2}.
\end{equation}
\item \label{el-tail2} The elliptic tail node $P$ does not belong to $C_{\rm exc}$.
\item \label{el-tail3} $E$ is either smooth or it is a rational (irreducible) curve with one node $Q$ or it is formed by two smooth rational curves
$R_1$ and $R_2$ meeting in two points $Q_1$ and $Q_2$, as depicted in Figure \ref{Fig-tails}.

%{\bf ADD a PICTURE! Type (1), (2) and (3)}
\end{enumerate}
\end{Lemma}
\begin{proof}
The equation \eqref{deg-tail} follows from the basic inequality \eqref{basic} applied to the subcurve
$E\subseteq C$ together with the fact that the inequalities must be strict since clearly
$E\not\subseteq C_{\rm ex}$ and $L$ is stably balanced by the hypothesis that $(C,L)\in \pdst$.
The fact that $d\not\equiv g-1 \mod (2g-2)$ follows from the fact that the there exists an integer $d_E$
satisfying the strict inequalities in \eqref{deg-tail}. This prove part \eqref{el-tail1}.

Next we turn to Part \eqref{el-tail2}. By contradiction, if $P\in R$ where $R$ is an exceptional component of $C$,
then $R^c$ is a disjoint union of two subcurves of $C$ each of which contains some components of $\w{C}$.
Therefore $\w{C}$ is disconnected and this contradicts the hypothesis that $(C,L)\in \pdst$
by Lemma \ref{compa-bal}.

Part \eqref{el-tail3} follows from the fact that $C$ is quasi-stable together with
Part \eqref{el-tail2} and the fact that $E^{\rm st}$ is either a smooth elliptic curve (which occurs
for Type I) or a rational irreducible curve with one node $Q$ (which occurs for Types II and III).
\end{proof}

%We will denote by $\nu:C^{\rm \nu}\to C$ the normalization morphism and we set $E^{\nu}:=\nu^{-1}(E)$ and
%$P^{\nu}:=\nu^{-1}(P)\cap E^{\nu}$.
%We choose the point $P^{\nu}$ as the identity of the group structure on the elliptic curve $E^{\nu}$.

The elements $(C,L)\in \pdgbar$ such that $C$ has an elliptic tail $E$ have special automorphisms
that will play a key role in the sequel.

\begin{Definition}\label{def-auto-tail}
Given an element $(C,L)\in \pdgbar$ such that $C$ has an elliptic tail $E$, an  automorphism $\phi=(\sigma,\psi)\in \Aut(C,L)$
(or its image in $\ov{\Aut(C,L)}$) is called an \emph{elliptic tail automorphism} of $(C,L)$ of order $n\geq 1$
(with respect to the elliptic tail $E\subseteq C$) if $\sigma$ is the identity on $\ov{C\setminus E}$ and
$\sigma_{|E}$ has order $n$.
\end{Definition}

The assumption that $(C,L)$ belongs to $\pdst$ puts some constraints on the possible elliptic tail automorphisms
that can occur. Indeed, under this assumption, using that the map $\ov{G}:\ov{\Aut(C,L)}\to \Aut(C^{\rm st})$
is injective (see Corollary \ref{Cor-auto}), we deduce immediately the following

\begin{Remark}\label{rmk-aut-tail}
Given $(C,L)\in \pdst$,
an element $\phi\in \Aut(C,L)$ is an elliptic tail automorphism of order $n\geq 1$ with
respect to the elliptic tail $E\subseteq C$ if and only if $G(\phi)\in \Aut(C^{\rm st})$ (see the notation of
Lemma \ref{automo}) is an elliptic tail automorphism of order $n$ of $C^{\rm st}$ with respect to the elliptic tail
$E^{\rm st}\subseteq C^{\rm st} $, i.e. $G(\phi)$ is the identity on $\ov{C^{\rm st}\setminus E^{\rm st}}$
and $G(\phi)_{|E^{\rm st}}$ has order $n$.
\end{Remark}

Using this Remark, we can give a complete description of the possible elliptic tail automorphisms
of elements $(C,L)\in \pdst$.

\begin{Lemma}\label{auto-ell}
Assume that $(C,L)\in \pdst$ and that $C$ has an elliptic tail $E\subseteq C$.
For an elliptic tail automorphism $\phi=(\sigma,\psi)\in \Aut{(C,L)}$ of
$(C,L)$ of order $n>1$ with respect to the elliptic tail $E\subseteq C$, the restriction
$\sigma_{|E}$ of $\sigma$ to $E$ must satisfy the following conditions
(according to whether the elliptic tail $E$ is of Type I, II or III as in Lemma \ref{tail-stable}\eqref{auto-ell3}):
\begin{enumerate}[(i)]
\item \label{auto-ell1} Type I: $\sigma_{|E}$ is an automorphism of $E$ fixing $P$ and $n=2$ or $n=4$ (which can  occur if and only if
$E$ has $j$-invariant equal to $1728$) or $n=3, 6$ (which can occur if and only if $E$ has $j$-invariant equal to $0$).
\item \label{auto-ell2} Type II: $\sigma_{|E}$ is an automorphism of order $n=2$ fixing $P$ and $Q$. If we
call $\nu: E^{\nu}\to E$ the normalization map and identify $E^{\nu}$ with $\P^1$ in such a way that
$\nu^{-1}(P)=\infty$ and $\nu^{-1}(Q)=\{1, -1\}$, then the automorphism $\sigma_{|E}$ is induced by the automorphism
$x\mapsto -x$ on $\P^1$.
\item \label{auto-ell3} Type III: $\sigma_{|E}$ is an automorphism of order $n=2$ such that, if we identify $R_i$ (for $i=1,2$) with $\P^1$ in such a way that
$Q_1$ and $Q_2$ get identified with $1$ and $-1$ (on both copies of $\P^1$) and $P\in R_1$ gets identified with $\infty$,
then $\sigma_{|R_i}$ (for $i=1, 2$) is equal to the automorphism $x\mapsto -x$ on $\P^1$. In particular,
$\sigma_{|E}$ fixes $P$ and exchanges $Q_1$ with $Q_2$.
\end{enumerate}
\end{Lemma}

\begin{proof}
Parts \eqref{auto-ell1} and \eqref{auto-ell2} follow easily from Remark \ref{rmk-aut-tail} together with
the fact that $E^{\rm st}\cong E$ for Type I and II and the well-known description of the elliptic tail
automorphisms of stable curves (see e.g. \cite[Rmk. 4.2.2]{Lud}).

In order to prove Part \eqref{auto-ell3}, observe that in this case $E^{\rm st}\subseteq C^{\rm st}$
is a rational curve with one node. Therefore, there exists a unique elliptic tail automorphism
of $C^{\rm st}$ with respect to $E^{\rm st}$, namely the automorphism $\sigma$ whose restriction
$\sigma_{|E}$ is described in Part \eqref{auto-ell2}. We conclude by Remark \ref{rmk-aut-tail}
together with the fact the elliptic tail automorphism of $(C,L)$ described in Part \eqref{auto-ell3}
is the unique (by Corollary \ref{Cor-auto}) lift to $\ov{\Aut(C,L)}$ of the elliptic tail automorphism
of $C^{\rm st}$ with respect to $E^{\rm st}$ described in Part \eqref{auto-ell2}.
\end{proof}

%We will often identify the elliptic tail automorphisms of $C$ with their images in $\Aut(C^{\rm st})$.

We can now determine the singular locus of $\pdst$.

\begin{Proposition}\label{sing}
The singular locus of $\pdst$ (for $g\geq 4$) is exactly the locus of pairs $(C,L)$ such that
$\ov{\Aut(C,L)}$ is not trivial.
\end{Proposition}
\begin{proof}
Near a point  $(C,L)\in \pdst$, using the local description (\ref{localring2}) and Corollary \ref{Cor-auto}, the scheme
$\pdst$ is isomorphic to the finite quotient
$$T_{\Def_{(C,L)}}/\ov{\Aut(C,L)},$$
where  $T_{\Def_{(C,L)}}$ is a $\C$-vector space of dimension $4g-3$ (by Proposition \ref{semiuniv}\eqref{semiuniv2})
and $\ov{\Aut(C,L)}$ can be naturally identified with a finite subgroup of $\GL(T_{\Def_{(C,L)}})$ .

%Consider the local description (\ref{localring2}) for a point $(C,L)\in \pdst$.
By a well-known result of Prill (see \cite{Prill}), it is enough to prove that
$\ov{\Aut(C,L)}\subseteq \GL(T_{\Def_{(C,L)}})$ does not contain quasi-reflections, i.e. elements $\phi$ such that $1$ is an eigenvalue
of $\phi$ with multiplicity equal to $4g-4$ or, equivalently, such that the
fixed locus ${\rm Fix}(\phi)$ of $\phi$ is a divisor inside $T_{\Def_{(C,L)}}$.

Consider the morphism $\phi_d:\pdgbar\to \mgbar$ which, according to \eqref{mor-tang}, locally looks like
$$T_{\Def_{(C,L)}}/\ov{\Aut(C,L)}\to T_{\Def_{C^{\rm st}}}/\Aut(C^{\rm st}),$$
where $T_{\Def_{(C,L)}}\twoheadrightarrow T_{\Def_{C^{\rm st}}}$ is surjective with kernel $V$ of dimension $g$
and $\Aut(C^{\rm st})$ can be naturally identified with a finite subgroup of $\GL(T_{\Def_{C^{\rm st}}})$.

%$\ov{\Aut(C,L)}\subseteq \Aut(C^{\rm st})$ by Corollary \ref{Cor-auto}.

Assume, by contradiction, that $\phi\in \ov{\Aut(C,L)}\subseteq \GL(T_{\Def_{(C,L)}})$ is a quasi-reflection.  By the above local description of the morphism
$\phi_d$, there are two possibilities for the image $\ov{G}(\phi)$ of $\phi$ in
$\Aut(C^{\rm st})\subseteq \GL(T_{\Def_{C^{\rm st}}})$ via the homomorphism $\ov{G}$ of Lemma \ref{automo}:
\begin{enumerate}[(i)]
\item $1$ is an eigenvalue of multiplicity $3g-3$ for $\ov{G}(\phi)$, i.e. $\ov{G}(\phi)=\id\in \Aut(C^{\rm st})$;
\item $1$ is an eigenvalue of multiplicity $3g-4$ for $\ov{G}(\phi)$, i.e. $\ov{G}(\phi)$ is a quasi-reflection for $\Aut(C^{\rm st})\subseteq \GL(T_{\Def_{C^{\rm st}}})$.
\end{enumerate}

In case (i), we conclude that $\phi=\id\in \Aut(C,L)$ since $\ov{G}$ is injective for an element $(C,L)\in \pdst$ by Corollary \ref{Cor-auto}. This contradicts
the fact that $\phi$ is a quasi-reflection.

In case (ii), it is well-known (see e.g. \cite[Cor. 4.2.6]{Lud}) that $C^{\rm st}$ must  have an elliptic tail $E$ and $\ov{G}(\phi)$ must be equal to the elliptic tail automorphism $i$ of
$C^{\rm st}$ of order $2$ with respect to $E$ (see Lemma \ref{auto-ell}).
Since $i=\ov{G}(\phi)$ admits a lifting to $\ov{\Aut(C,L)}$, namely $\phi$,  the line bundle $L$ on $C$ must be such that the restriction $L_{|E}$ of $L$ to $E$
is a suitable translate of a $2$-torsion point of $\Pic^0(E)$ (using some identification $\Pic^{d_E}(E)\cong \Pic^0(E)$ and the fact that $i$ acts on $\Pic^0(E)$
sending $\eta$ into $\eta^{-1}$).
Therefore,  the fixed locus ${\rm Fix}(\phi)$ of $\phi$ inside $T_{\Def_{(C,L)}}$ has codimension at least two, hence $\phi$ is not a quasi-reflection.
%Clearly a quasi-reflection
%in $\ov{\Aut(C,L)}$ maps to a quasi-reflection of $\Aut(C^{\rm st})$ (via the map
%of Lemma \ref{automo}). It is well-known (see for example \cite[Cor. 4.2.6]{Lud}) that the quasi-reflections of
%$\Aut(C^{\rm st})$ are the elliptic tails involutions $i$ of $C^{\rm st}$ (see notation \ref{Ell-tail}).
%Clearly, $i$ fixes $L$ if and only if $(\nu^*L)_{|E^{\nu}}(-d_E \cdot \widetilde{p})$ is
%a $2$-torsion line bundle on $\Pic^0(E^{\nu})\cong E^{\nu}$.
%Therefore the fixed locus of $i$ inside $\un{\Def}_{(C,L)}$ has codimension at least two, hence $i$ is not a quasi-reflection.
\end{proof}

By applying the Reid--Tai criterion for canonical singularities, we can prove the
following result.

\begin{Theorem}\label{sing-Pdst}
Assume $g\geq 4$. Then the stable locus $\pdst$ has canonical singularities.
In particular, if  $\w{\pdst}$ is a resolution of singularities of $\pdst$, then
every pluricanonical form defined on the smooth locus $(\pdst)^{\rm reg}$
of $\pdst$ extends holomorphically to $\w{\pdst}$, that is,
for all  integers $m$ we have
$$h^0\left((\pdst)^{\rm reg},m K_{(\pdst)^{\rm reg}}\right)=h^0\left(\w{\pdst}, m K_{\w{\pdst}}\right).
$$

\end{Theorem}
\begin{proof}
We use the notation introduced in the proof of Proposition \ref{sing}.

%Near a point  $(C,L)\in \pdst$, using the local description (\ref{localring2}) and Corollary \ref{Cor-auto}, the scheme
%$\pdst$ is isomorphic to the finite quotient singularity
%$$T_{\Def}_{(C,L)}}/\ov{\Aut(C,L)},$$
%where  $T_{\Def}_{C,L}}$ is a $\C$-vector space of dimension $4g-3$ and $\ov{\Aut(C,L)}$ can be naturally identified
%with a finite subgroup of $\GL(T_{\Def}_{(C,L)}})$ . Moreover, by (\ref{mor-tang}), the morphism $\phi_d:\pdgbar\to \mgbar$
%looks locally like
%$$T_{\Def}_{(C,L)}}/\ov{\Aut(C,L)}\to T_{\Def}_{C^{\rm st}}}/\Aut(C^{\rm st}),$$
%where $T_{\Def}_{(C,L)}}\twoheadrightarrow T_{\Def_{C^{\rm st}}}$ is surjective with kernel $V$ of dimension $g$.
%We fix a splitting $T_{\Def_{(C,L)}}=V\oplus  T_{\Def_{C^{\rm st}}}$. Since $\ov{\Aut(C,L)}$ is a subgroup of
%$\Aut(C^{\rm st})$ by Corollary  \ref{Cor-auto},
%we have that $\ov{\Aut(C,L)}\hookrightarrow \GL(V)\oplus\GL(T_{\Def_{C^{\rm st}}})\hookrightarrow \GL(T_{\Def_{(C,L)}})$.

Given an element $\phi\in \ov{\Aut(C,L)}\subseteq \GL(T_{\Def_{(C,L)}})$ of order $n$, we can choose
suitable coordinates of $T_{\Def_{(C,L)}}$ and a primitive $n$-th root of unity $\zeta$,
such that the action of  $\phi$ on $T_{\Def_{(C,L)}}$ is given by the sum $M(\phi)\oplus N(\phi)$ of two matrices
(with $0\leq a_i< n$ for $1\leq i\leq 4g-3$):
%$T_{\Def_{C^{\rm st}}}$ and $V$ are given, respectively, by
$$M(\phi)=\left(
  \begin{array}{ccc}
    \zeta^{a_1} &  & 0 \\
     & \ddots &  \\
    0 &  & \zeta^{a_{3g-3}} \\
  \end{array}
\right)
\hspace{0,2cm}\text{ and }\hspace{0,2cm}
N(\phi)=\left(
  \begin{array}{ccc}
    \zeta^{a_{3g-2}} &  & 0 \\
     & \ddots &  \\
    0 &  & \zeta^{a_{4g-3}} \\
  \end{array}
\right)
$$
in such a way that the action of $\phi$ on $V$ is given by $N(\phi)$ and the action of $\ov{G}(\phi)$ on $T_{\Def_{C^{\rm st}}}$ is given by $M(\phi)$.
%Clearly the action of $\phi$ on $T_{\Def_{(C,L)}}$ is given by the sum $M(\phi)\oplus N(\phi)$.

Recall that, according to the Reid--Tai criterion
for the canonicity of finite quotient singularities (see e.g. \cite[pp. 27--28]{HarMum}
or \cite[Thm. 4.1.11]{Lud}),
a point $(C,L)$ is a canonical singularity if and only if for every $\phi\in \ov{\Aut(C,L)}$
of some order $n$ and every $n$-th root of unity $\zeta$ we have
\begin{equation}\label{RT-cond}
\sum_{i=1}^{4g-3}\frac{a_i}{n}\geq 1.
\end{equation}
Note that this is true because $\ov{\Aut(C,L)}$ does not contain quasi-reflections
(see the proof of Proposition \ref{sing}).

Denote, as usual, by $\Delta_1$ the divisor of $\ov{M}_g$ consisting of curves
having an elliptic tail. If  $C^{\rm st}\not\in \Delta_1$ or $C^{\rm st}\in \Delta_1$ but $\ov{G}(\phi)$ is not  an elliptic tail automorphism
(or equivalently, by Remark \ref{rmk-aut-tail}, $\phi$ is not an elliptic tail automorphism)
%(see Notation \ref{Ell-tail}),
then by \cite[Thm. 2]{HarMum} we get
$$\sum_{i=1}^{4g-3}\frac{a_i}{n}\geq \sum_{i=1}^{3g-3}\frac{a_i}{n}\geq 1,$$
and we are done in this case.

If $C^{\rm st}\in \Delta_1$ and $\ov{G}(\phi)$ is an elliptic tail automorphism with respect to the elliptic tail
$E^{\rm st}\subset C^{\rm st}$ (where $E^{\rm st}$ is equal to the image via ${\rm st}:C\to C^{\rm st}$ of the elliptic
tail $E\subset C$ as in Remark \ref{rmk-ell-tail})
then we choose, as in \cite[Prop. 4.2.5]{Lud},
the first two coordinates $t_1$ and $t_2$ of $T_{\Def_{C^{\rm st}}}$ in such a way that (in the notation of Lemma \ref{tail-stable}):
$t_1$ corresponds to the elliptic tail node $P$ and $t_2$ correspond to $Q$ if
$E^{\rm st}$ is singular and is a coordinate for $T_{(E,P)}(M_{1,1})$ if $E$ is smooth.
%Moreover, we choose the first coordinate $s_1$ on $V$ so that it is a coordinate
%for $T_{L_{|E}}(\Pic^{d_E}(E))$, where $d_E$ is defined in Lemma \ref{tail-stable}\eqref{el-tail1}.
% where $L_E:=L_{|E}$.
In \cite[Prop. 4.2.5]{Lud}, it is proved that the matrix $M(\phi)$ is given by
(depending on the choice of the primitive $n$-th root of unity $\zeta$):
\begin{equation}\label{M-matrix}
M(\phi)=\begin{cases}
\left(
  \begin{array}{ccc}
    \zeta^{1} &  &  \\
     & \zeta^0 &  \\
     &  & \II  \\
  \end{array}
\right)  & \text{ if } n=2,\\
\left(
  \begin{array}{ccc}
    \zeta^{1} &  &  \\
     & \zeta^2 &  \\
     &  & \II  \\
  \end{array}
\right)  \text{ or }
\left(
  \begin{array}{ccc}
    \zeta^{3} &  &  \\
     & \zeta^2 &  \\
     &  & \II  \\
  \end{array}
\right)
& \text{ if } n=4,\\
\left(
  \begin{array}{ccc}
    \zeta^{1} &  &  \\
     & \zeta^2 &  \\
     &  & \II  \\
  \end{array}
\right)  \text{ or }
\left(
  \begin{array}{ccc}
    \zeta^2 &  &  \\
     & \zeta^1 &  \\
     &  & \II  \\
  \end{array}
\right)
& \text{ if } n=3,\\
\left(
  \begin{array}{ccc}
    \zeta^5 &  &  \\
     & \zeta^4 &  \\
     &  & \II  \\
  \end{array}
\right)  \text{ or }
\left(
  \begin{array}{ccc}
    \zeta^1 &  &  \\
     & \zeta^2 &  \\
     &  & \II  \\
  \end{array}
\right)
& \text{ if } n=6,\\
\end{cases}
\end{equation}
where $\II$ is the suitable unit matrix.

Let us now turn to the matrix $N(\phi)$. We choose the first coordinate $s_1$ on $V$ so that it is a coordinate
for $T_{L_{|E}}(\Pic^{d_E}(E))$, where $d_E$ is defined in Lemma \ref{tail-stable}\eqref{el-tail1}.
In order to compute the action of $\phi$ on $s_1$,
we distinguish three case according to whether the elliptic
tail $E\subset C$ is of Type I, II or III (see Lemma \ref{tail-stable}\eqref{el-tail3} and Figure \ref{Fig-tails}).

If $E$ is of Type I, i.e. $E$ is smooth, then we can identify $E$ with $\Pic^{d_E}(E)$ sending $q\in E$ into
$\O_{E}(q+(d_E-1)P)\in \Pic^{d_E}(E)$.
Since $\phi$ acts on $\Pic^{d_E}(E)$ via pull-back, if
%The action of $\phi$ on $\Pic^{d_E}(E)$ is given by sending $L_{|E}$ to $\ov{G}(\phi)^*(L)_{|E}$. Therefore if
the action of $\ov{G}(\phi)$ on $T_P(E)$ is given by the multiplication by a root of unity $\zeta$, then the action of $\phi$ on $T_{L_{|E}}(\Pic^{d_E}(E))$
is given by the multiplication by $\zeta^{-1}$.
%If we identify $E$ with $\Pic^{d_E}(E)$ (by sending $q\in E$ into $\O_E(q+(d_E-1) P)$), then the action of
%$\phi$ on $\Pic^{d_E}(E)$ is induced by the pull-back map $\ov{G}(\phi)_{|E}^*$.
%On the other hand, since we choose $P^{\nu}$ as the origin of the group structure on
%$E^{\nu}$, if we identify $\Pic^{d_E}(E^{\nu})$ with $\Pic^{0}(E^{\nu})\cong E^{\rm reg}$
%(via tensorization with $\O_E(-d_E\cdot P^{\nu})$),
Therefore the matrix $N(\phi)$ is equal to (with respect to the same choice of the primitive $n$-th root of unity $\zeta$ as in the above
matrix $M(\phi)$):
\begin{equation}\label{N-matrix}
N(\phi)=\begin{cases}
\left(
  \begin{array}{cc}
    \zeta^{1} &    \\
     &   \II  \\
  \end{array}
\right)  & \text{ if } n=2,\\
\left(
  \begin{array}{cc}
    \zeta^{3} &   \\
     &   \II  \\
  \end{array}
\right)  \text{ or }
\left(
  \begin{array}{cc}
    \zeta^{1} &   \\
     &   \II  \\
  \end{array}
\right)
& \text{ if } n=4,\\
\left(
  \begin{array}{cc}
    \zeta^{2} &   \\
     &   \II  \\
  \end{array}
\right)  \text{ or }
\left(
  \begin{array}{cc}
    \zeta^1 &    \\
     & \II  \\
  \end{array}
\right)
& \text{ if } n=3,\\
\left(
  \begin{array}{cc}
    \zeta^1 &   \\
     &  \II  \\
  \end{array}
\right)  \text{ or }
\left(
  \begin{array}{cc}
    \zeta^5 &   \\
     &  \II  \\
  \end{array}
\right)
& \text{ if } n=6.\\
\end{cases}
\end{equation}

If $E$ is of type II, i.e. $E$ is an irreducible rational curve with one node $Q$ (as in Figure \ref{Fig-tails}), then $\Pic^{d_E}(E)\cong \Gm$.
Explicitly, if we consider the normalization morphism $\nu:E^{\nu}\cong \P^1\to E$ and let $\nu^{-1}(Q)=\{u,v\}$, then
any $\lambda\in \Gm(k)$ determines a unique line bundle $L_{\lambda}\in \Pic^{d_E}(E)$ whose local sections
are the local sections $s$ of $\O_{\P^1}(d_E)$ such that $s(u)=\lambda s(v)$.
Since $\phi_{|E}$ is induced by an involution of $E^{\nu}$ that exchanges $u$ and $v$ (by Lemma \ref{auto-ell}\eqref{auto-ell2}),
then clearly $\phi$ will send $L_{\lambda}$ into $L_{\lambda^{-1}}$.  This implies that the action of $\phi$ on
$T_{L_{|E}}(\Pic^{d_E}(E))$ is given by multiplication by $-1$, hence the matrix $N(\phi)$ is also in this case given
by  \eqref{N-matrix} with $n=2$.

If $E$ is of type III, i.e. $E$ is made of two irreducible rational components $R_1$ and $R_2$ meeting in two points $Q_1$ and $Q_2$
(as in Figure \ref{Fig-tails}), then again $\Pic^{d_E}(E)\cong \Gm$.
Explicitly, if we consider the normalization morphism $\nu:E^{\nu}=R_1\coprod R_2 \to E$ and let $\nu^{-1}(Q_i)=\{u_i,v_i\}$
with $u_i\in R_1$ and $v_i\in R_2$ (for $i=1,2$), then
any $\lambda\in \Gm(k)$ determines a unique line bundle $L_{\lambda}\in \Pic^{d_E}(E)$ whose local sections
are pairs of local sections $(s_1,s_2)$ of $(\O_{R_1}(d_E-1), \O_{R_2}(1))$ such that
$$\frac{s_1(u_1)}{s_1(u_2)}=\lambda \frac{s_2(v_1)}{s_2(v_2)}.$$
Since $\phi_{|E}$ is induced by an involution of $E^{\nu}$ that exchanges $u_1$ with $u_2$ and $v_1$ with $v_2$ (by Lemma \ref{auto-ell}\eqref{auto-ell3}),
then clearly $\phi$ will send $L_{\lambda}$ into $L_{\lambda^{-1}}$.  This implies that the action of $\phi$ on
$T_{L_{|E}}(\Pic^{d_E}(E))$ is given by multiplication by $-1$, hence the matrix $N(\phi)$ is also in this case given
by  \eqref{N-matrix} with $n=2$.

An easy inspection of the matrices $M(\phi)$ in \eqref{M-matrix} and $N(\phi)$ in \eqref{N-matrix} reveals that the condition
(\ref{RT-cond}) is always satisfied, which shows that $\pdst$ has canonical singularities.

The last assertion of the theorem follows from the well-known fact that canonical singularities
do not impose adjoint conditions on the pluricanonical forms.
\end{proof}

\section{The canonical class of $\Pdgbar$ and of $\pdgbar$.}\label{canonical}

The aim of this section is to prove Theorem \ref{K}.
To achieve that, we first determine the canonical class of the stack $\Pdgbar$.

\begin{Theorem}\label{can-stack}
The canonical class of $\Pdgbar$ is equal to
$$K_{\Pdgbar}=\Phi_d^*(14\lambda-2\delta),$$
where $\lambda$ and $\delta$ are the Hodge and total boundary class on $\Mgbar$.
\end{Theorem}
\begin{proof}
Let $\pi: \Pdguniv \rightarrow \Pdgbar$ the universal family over $\Pdgbar$ and $\LL_d$ the universal line bundle
over $\Pdguniv$ (see \cite{Melo2} for a modular description of $\Pdguniv$).
Denote by $\Omega_{\pi}$ and $\omega_{\pi}$ the sheaf of relative K\"{a}hler differentials and the relative dualizing sheaf, respectively.
Let $d: {\mathcal O}_{\Pdguniv}^* \rightarrow \Omega_{\pi}$ be the universal derivation and consider the map induced in cohomology
$\theta: \Pic({\Pdguniv}) \rightarrow H^1(\Pdguniv, \Omega_{\pi})$. Since
$H^1({\Pdguniv}, \Omega_{\pi}) \cong \Ext^1({\mathcal O}_{{\Pdguniv}},\Omega_{\pi}),$
the map $\theta$ sends the line bundle ${\LL_d}$ on ${\Pdguniv}$ into the class of an extension
\begin{equation}\label{Atiya}
0 \rightarrow \Omega_{\pi} \rightarrow \mathcal{E} \rightarrow {\mathcal O}_{{\Pdguniv}}\rightarrow 0.
\end{equation}
The restriction of the above extension (\ref{Atiya}) to a geometric fiber $(C,L)$ of $\pi$ is the extension (\ref{jets})
as it follows from the discussion in Section \ref{loc-struc}.

>From this and the analysis of the deformation theory of the pair $(C,L)$ carried out in \ref{loc-struc}, it follows that
the tangent space of $\pdstack$ at a geometric point $(C,L)$ is equal to $\Ext^1(\EE_{|C},\OO_{C})-\Ext^0(\EE_{|C},\OO_C)$
Therefore, using relative duality for $\pi$,  it follows that
the canonical class $K_{\Pdgbar}$ of $\Pdgbar$  is equal to
$$K_{\Pdgbar}=c_1\left(\pi_!\left(\EE\otimes \omega_{\pi}\right)\right).$$
To compute this class, we  apply the Grothendieck-Riemann-Roch Theorem for quotient stacks (\cite{EG}) relative to the morphism $\pi$:
\begin{equation}
\label{GRR}
\ch\left(\pi_!\left(\omega_{\pi} \otimes \mathcal{E}\right)\right)=\pi_*\left(\ch(\omega_{\pi} \otimes \mathcal{E})\cdot \Td(\Omega_{\pi})^{-1}\right).
\end{equation}

Let us compute the degree one part of the right hand side of (\ref{GRR}).
We set $\w{K}:=c_1(\omega_{\pi})$
and $\w{\eta}:=c_2(\Omega_{\pi})$. Note that, as remarked in \cite[p. 158]{HM}, we have $\w{K}=c_1(\Omega_{\pi})$.

The first three terms of inverse of the Todd class of $\Omega_{\pi}$ are equal to
\begin{equation}\label{Todd}
\Td(\Omega_{\pi})^{-1}=1-\frac{c_1(\Omega_{\pi})}{2}+\frac{c_1^2(\Omega_{\pi})+c_2(\Omega_{\pi})}{12} + \ldots=
1-\frac{\w{K}}{2}+\frac{\w{K}^2+\w{\eta}}{12}+\ldots
\end{equation}
Using (\ref{Atiya}), we get
$$\ch(\EE\otimes \omega_{\pi})=\ch(\EE)\cdot \ch(\omega_{\pi})=(\ch(\Omega_{\pi})+\ch(\OO_{\Pdguniv}))\cdot \ch(\omega_{\pi})=$$
\begin{equation}\label{ch1}
=(\ch(\Omega_{\pi})+1)\cdot \ch(\omega_{\pi}).
\end{equation}
Moreover we have
$$\ch(\omega_{\pi})=1+c_1(\omega_{\pi})+\frac{c_1(\omega_{\pi})}{2}+\ldots=1+\w{K}+\frac{\w{K}^2}{2}+\ldots$$
$$\ch(\Omega_{\pi})=1+c_1(\Omega_{\pi})+\frac{c_1(\Omega_{\pi})}{2}-c_2(\Omega_{\pi})+\ldots=1+\w{K}+\frac{\w{K}^2}{2}-\w{\eta}+\ldots$$
Substituting into (\ref{ch1}), we arrive at
\begin{equation}\label{ch2}
\ch(\EE\otimes \omega_{\pi})=2+3\w{K}+\frac{5}{2}\w{K}^2-\w{\eta}+\ldots
\end{equation}
Combining (\ref{Todd}) and (\ref{ch2}), we get
$$\left[\ch(\omega_{\pi} \otimes \mathcal{E})\cdot \Td(\Omega_{\pi})^{-1}\right]_2= \frac{\w{K}^2+\w{\eta}}{6}-\frac{3}{2}\w{K}^2+\frac{5}{2}\w{K}^2-\w{\eta}=
\frac{7}{6}\w{K}^2-\frac{5}{6}\w{\eta},
$$
hence, from (\ref{GRR}), we deduce
\begin{equation}\label{Kpi}
K_{\Pdgbar}=\frac76 \pi_*(\w{K}^2) - \frac56 \pi_*(\w{\eta}).
\end{equation}
Let us now apply the Grothendieck-Riemann-Roch theorem to the sheaf $\omega_{\pi}$. Since $R^1\pi_*\omega_{\pi}=\OO_{\Pdgbar}$ by relative duality,
we get
$$
c_1(\pi_*\omega_{\pi})=\pi_*\left(\frac{\w{K}^2+\w{\eta}}{12}-\frac12 \w{K}^2 + \frac12\w{K}^2\right)=\frac{1}{12} \pi_*(\w{K}^2)+ \frac{1}{12}\pi_*(\w{\eta}).
$$
If we set $\w{\lambda}:=c_1(\pi_*\omega_{\pi})$ and $\w{\delta}:=\pi_*(\w{\eta})$, then  the previous relation becomes
$12\w{\lambda}=\pi_*(\w{K}^2)+\w{\delta}.$
Substituting into (\ref{Kpi}), we obtain
\begin{equation*}
K_{\Pdgbar}=14 \w{\lambda}-2 \w{\delta}.
\end{equation*}
The Lemma below completes the proof.
\end{proof}

\begin{Lemma}
With the notation of Theorem \ref{can-stack}, we have
$$\w{\lambda}=\Phi_d^*(\lambda) \text{ and } \w{\delta}=\Phi_d^*(\delta). $$
\end{Lemma}
\begin{proof} Consider the diagram
\begin{equation}\label{univ-fam}
\xymatrix{
\Pdguniv \ar[r]^{\Phi_{d,1}} \ar[d]_{\pi} & \Mguniv \ar[d]^{\ov{\pi}}\\
\Pdgbar \ar[r]^{\Phi_d} &  \Mgbar \\
}
\end{equation}
Recall that the classes $\lambda$ and $\delta$ on $\Mgbar$ are defined as
$$\lambda:=c_1(\ov{\pi}_*(\omega_{\ov{\pi}})) \text{ and }�\delta:=\ov{\pi}_*(c_2(\Omega_{\ov{\pi}})),$$
where $\Omega_{\ov{\pi}}$ and $\omega_{\ov{\pi}}$ are the sheaf of relative K\"{a}hler differentials and the relative dualizing sheaf of $\ov{\pi}$,
respectively.

The map $\Phi_d$ sends an element  $(\CC\to S, \LL)\in \Pdgbar(S)$ into the stabilization $\CC^{\rm st}\to S\in \Mgbar(S)$.
Recall that for every quasi-stable (or more generally semistable) curve $C$ with stabilization morphism $\psi:C\to C^{\rm st}$,
the pull-back via $\psi$ induces a natural isomorphism $\psi^*:H^0(C^{\rm st},\omega_{C^{\rm st}})\stackrel{\cong}{\to} H^0(C,\omega_C)$.
Therefore we have $\Phi_d^*(\ov{\pi}_*(\omega_{\ov{\pi}}))=\pi_*(\omega_{\pi})$  and, by taking the first Chern classes, we get $\w{\lambda}=\Phi_d^*(\lambda)$.

On the other hand, since the class $\w{\delta}$ is the total boundary class of $\Pdgbar$ and $\delta$ is the total boundary class
of $\Mgbar$ (see \cite[pp. 49--50]{HarMum}), it is clear that $\w{\delta}=\Phi_d^*(\delta)$.
\end{proof}

\begin{proof}[Proof of Theorem \ref{K}.]

Let $p:\Pdgbar\to \pdgbar$ the natural map from the stack $\Pdgbar$ to its good moduli space.
In view of Theorem \ref{can-stack}, it is enough to show that $p^*(K_{\pdgbar})=K_{\Pdgbar}$.

Clearly, the two classes agree on the interior $\Pdg$, since for $C$ varying in an open subset of $M_g$
whose complement has codimension at least two (since $g\geq 4$), by Lemma \ref{automo} we have
that $\ov{\Aut(C,L)}\subseteq \Aut(C)=\{\id\}$.

Let us now look at the boundary of $\pdgbar$. By \cite[Prop. 4]{Fon}, the boundary of $\pdgbar$ is the union of the irreducible
divisors $D_i:=\phi_d^{-1}(\Delta_i)$, for $i=0,\cdots,[g/2]$. Let $k_{d,g}:=(2g-2, d-g+1)$.
A general element of $(C,L)\in D_i$ looks as follows (see e.g \cite[Ex. 7.1, 7.2]{Cap} and \cite[Prop. 2]{Melo}):
\begin{enumerate}
\item \label{case1}
If $i=0$ then $C$ is a general irreducible nodal curve with one node and $L$ is a general line bundle of degree $d$ on $C$;
\item \label{case2}
If $i>0$ and $2g-2$ does not divide $(2i-1)\cdot k_{d,g}$, then $C$ is a stable curve consisting of two general smooth curves
$C_1$ and $C_2$ of genera, respectively,  $i$ and $g-i$ meeting in one point and $L$ is a general line bundle of multidegree
$(\deg_{C_1}(L),\deg_{C_2}(L))=(d_1,d_2=d-d_1)$ where $d_1$ is the unique integer such that
$$\left|d_1-\frac{d(2i-1)}{2g-2}\right|< \frac{1}{2}.$$
\item \label{case3}
If $i>0$ and $2g-2$ divides $(2i-1)\cdot k_{d,g}$, then $C$ is a quasi-stable curve consisting of two general smooth curves
$C_1$ and $C_2$ of genera, respectively,  $i$ and $g-i$ joined by a rational curve $R\cong \P^1$
and $L$ is a general line bundle whose  multidegree is such that $\deg_R L=1$ and
$$\left\{\begin{aligned}
& d_1:=\deg_{C_1}L=\frac{d(2i-1)}{(2g-2)}-\frac{1}{2},\\
& d_2:=\deg_{C_2} L= \frac{d(2g-2i-1)}{(2g-2)}-\frac{1}{2}.
 \end{aligned}\right.
 $$
 \end{enumerate}
We claim that the automorphism group of a general point $(C,L)\in D_i$ is equal to
\begin{equation}\label{auto-boun}
\Aut(C,L)=\begin{cases}
\Gm & \text{ in cases } \ref{case1} \text{ and }�\ref{case2}, \\
\Gm^2 & \text{ in case } \ref{case3}.
\end{cases}
\end{equation}
Indeed, by the explicit description  above, $\gamma(\w{C})=1$ in cases (\ref{case1}) and (\ref{case2}), and
$\gamma(\w{C})=2$ in case (\ref{case3}). Therefore, the claim will follow from Lemma \ref{automo} if we show
that the image of $\Aut(C,L)\to \Aut(C^{\rm st})$ is trivial. This is trivially true if $i\neq 1$ since in this case
$C^{\rm st}$ is a general curve in $\Delta_i$, hence $\Aut(C^{\rm st})=\{\id\}$. If $i=1$ then
$\Aut(C^{\rm st})=\Z/2\Z$  generated by the elliptic tail involution $\sigma$ with respect to the elliptic tail $C_1$
(see Remark \ref{rmk-aut-tail} and the notation there). However, in this case, $\sigma$ comes from an automorphism of the pair $(C,L)$ if and only
if $L_{|C_1}(-d_1\cdot P)$ is a $2$-torsion point of $\Pic^0(C_1)$, where $P=C_1\cap R$ is the elliptic tail node of $C_1$.
% i.e. the intersection of $C_1$ and $R$.
Clearly, this is not the case for a general strictly balanced line bundle $L$ on $C$.

In cases (\ref{case1}) and (\ref{case2}), $\un{\Def}_{(C,L)}$ has dimension  $4g-3$ and $\Aut(C,L)=\Gm$ acts trivially on it
(see \ref{loc-struc}). Therefore, the morphism $p$ looks locally at $(C,L)$ as
$\w{p}:[\un{\Def}_{(C,L)}/\Gm]=\un{\Def}_{(C,L)}\times B\Gm\to \un{\Def}_{(C,L)}$. It is clear that in this case
$\w{p}^*(K_{\un{\Def}_{(C,L)}})=K_{\un{\Def}_{(C,L)}\times B\Gm}$.

In case (\ref{case3}), $\un{\Def}_{(C,L)}$ has dimension $4g-2$ (see \ref{loc-struc}). If we choose the first two coordinates
$x$ and $y$ of $\un{\Def}_{(C,L)}$ in such a way that they correspond to the  local deformations of the two nodes
$P_1:=C_1\cap R$ and $P_2:=C_2\cap R$, then the action of $(\mu, \nu)\in \Aut(C,L)=\Gm^2$ on the first two coordinates
of $\un{\Def}_{(C,L)}$ is given by
\begin{equation}\label{exp-act}
(\mu, \nu)\cdot (x,y)=(\mu\nu^{-1}x, \mu^{-1}\nu y),
\end{equation}
while it is trivial on the other coordinates. Therefore, neglecting the trivial  coordinates, at $(C,L)$ the morphism $p$
looks locally as
$$\X:=[\Spf\C[[x,y]]/\Gm^2]\stackrel{\w{p}}{\longrightarrow} \Spf\C[[x,y]]/\Gm^2:=X.$$
Since the ring of invariants for the action (\ref{exp-act}) of $\Gm^2$ on $\C[[x,y]]$
is generated by $xy$, the quotient $X$ is isomorphic to $\Spf \C[[xy]]$.
The quotient map $Y:=\Spf\C[[x,y]]\stackrel{q}{\to} X=\Spf \C[[xy]]$ induces a pull-back map
$$\begin{aligned}
q^*:\langle d(xy)\rangle= \Omega^1_X& \to \Omega^1_Y=\langle dx, dy\rangle,\\
d(xy) & \mapsto xdy+ydx.
\end{aligned}
$$
On the other hand, the cotangent complex of $\X$ is equal to the
$\Gm^2$-equivariant cotangent complex of $Y:=\Spf\C[[x,y]]$  (see e.g. \cite[p. 37]{Moc}),
which in our case looks like:
$$
\L:  \langle dx, dy\rangle=\Omega^1_Y  \stackrel{f}{\to}  \OO_Y\otimes {\rm Lie}(\Gm^2)^*=\OO_Y\otimes \left\langle
\frac{d\lambda}{\lambda}, \frac{d\mu}{\mu}\right\rangle
$$
The map $f$ is the  dual of the infinitesimal action of ${\rm Lie}(\Gm^2)$ on $Y$, hence,
by the explicit action (\ref{exp-act}), we compute
$$\left\{\begin{aligned}
f(dx)= & x\frac{d\lambda}{\lambda}-x\frac{d\mu}{\mu},\\
f(dy)= & y\frac{d\mu}{\mu}-y\frac{d\lambda}{\lambda}.\\
\end{aligned}\right.
$$
Since the image of $q^*$ is equal to the kernel of $f$, we deduce that $\w{p}^*(K_X)=K_{\X}$, which concludes
the proof.
\end{proof}

\section{The Iitaka dimension of $K_{\pdgbar}$}\label{Iitaka}

The aim of this section is to prove Theorem \ref{Iitaka-Mg}.
Since $K_{\pdgbar}=\phi_d^*(14\lambda-2\delta)$ (by Theorem \ref{K}) and $\phi_d$ has connected fibers,
we have
\begin{equation}\label{ineq-Kod}
\kappa(\can)= \kappa(14\lambda-2\delta).
\end{equation}
Therefore, we are reduced to study the Iitaka dimension of the divisor $14\lambda-2\delta$
on $\mgbar$.  Notice that the slope of $14\lambda-2\delta$ is equal to $7$. Hence, if the slope $s(\mgbar)$ of $\mgbar$
(in the sense of Harris-Morrison \cite{HM90}) is strictly less than $7$ then we conclude that
$14\lambda-2\delta$ is big, i.e. that $\kappa(14\lambda-2\delta)=3g-3$; while if $s(\mgbar)>7$ then $14\lambda-2\delta$ is not pseudo-effective and
$\kappa(14\lambda-2\delta)=-\infty$ (see the discussion at the beginning of \cite{HM90}).

\begin{Proposition}\label{kod-infty}
If $g \le 9$ then $\kappa(K_{\pdgbar})=-\infty$.
\end{Proposition}
\begin{proof}
This follows by the fact that $s(\mgbar)>7$ for $g\leq 9$ (see \cite{Tan}).
\end{proof}

\begin{Remark}
By combining the above Proposition \ref{kod-infty} with the inequality \eqref{comp-Kod}, we obtain another proof of the fact that
$\kappa(P_{d,g})=-\infty$ for $g\leq 9$ and any $d$ (which of course follows from the stronger Theorem \ref{Ver-thm}).
\end{Remark}

For $g\geq 12$ we can prove the following

\begin{Proposition}\label{dim-Kod}
If $g\geq 12$ then $\kappa(\can)=3g-3$ and the fibration
$\phi_d:\pdgbar\to \mgbar$ is the Iitaka fibration of $\can$.
\end{Proposition}
\begin{proof}
We have already observed, in the proof of Proposition \ref{kod-fib},
that if $\kappa(\can)=3g-3$ then $\phi_d:\pdgbar\to \mgbar$ is the Iitaka fibration (see \cite[Def. 2.1.34]{Laz}) of $\can$.
Therefore, it is enough to prove the first assertion.
As pointed out before, this will follow if we show that $s(\mgbar)<7$ for $g\geq 12$.

By computing the class of the Brill-Noether divisor $D_{\frac{g+1}{2}}^1$, Harris and Mumford proved
in \cite{HarMum} that
$$s(\mgbar)\leq 6+\frac{12}{g+1} \text{ if } g \text{ is odd. }$$
Since $6+\frac{12}{g+1}<7$ if and only if $g>11$, we get that
\begin{equation}\label{equ1}
s(\mgbar)<7  \text{ if } g \text{ is odd and } g\geq 13.
\end{equation}

By computing the class of the Petri divisor $E_{\frac{g}{2}+1}^1$, Eisenbud and Harris
in \cite[Thm. 2]{EH} proved that
$$s(\mgbar)\leq 6+\frac{14g+4}{g(g+2)} \text{ if } g \text{ is even. }$$
Since $6+\frac{14g+4}{g(g+2)}<7$ if and only if $g>13$, we get that
\begin{equation}\label{equ2}
s(\mgbar)<7 \text{ if } g \text{ is even and } g\geq 14.
\end{equation}

By computing the slope of some effective divisors on $\mgbar$ associated to curves equipped with
secant-exceptional linear series, Cotterill in \cite[\S 6.2]{Cot} showed in particular that
\begin{equation}\label{equ3}
s(\overline{M}_{12})\leq  6,979...< 7.
\end{equation}

Equations (\ref{equ1}), (\ref{equ2}) and (\ref{equ3}) together imply the result.
\end{proof}

The cases $g=10$ and $g=11$ requires a special care since it is known that in this case
$s(\mgbar)=7$ (see \cite{Tan} and \cite[Cor. 1.3]{FarPop}). We start with the case $g=10$.

\begin{Proposition}\label{dim-Kod10}
If $g=10$ then $\kappa(\can)=0$.
\end{Proposition}
\begin{proof}
Farkas and Popa proved in \cite[Thm. 1.6]{FarPop} that the effective irreducible divisor $F$ (which is denoted by $\ov{K}$ in loc. cit.)
given by the closure of the locus of smooth curves of genus $10$ lying on a $K3$ surface
has class equal to
\begin{equation}\label{class-K}
F=7\lambda-\delta_0-5 \delta_1-9 \delta_2-12 \delta_3-14 \delta_4- B_5\delta_5,
\end{equation}
with $B_5\geq 6$. Since it is easily checked that $14\lambda-2\delta$ is the sum of $2 F$
and an effective boundary divisor, we get, using (\ref{ineq-Kod}), that $\kappa(\can)=
\kappa(14\lambda-2\delta)\geq \kappa(2F)=\kappa(F)\geq 0$.

It remains to prove that $h^0(\ov{M}_{10}, m(14\lambda-2\delta))=1$ for any $m$ sufficiently divisible.

\underline{Claim 1}: If $m$ is sufficiently divisible then
$2mF$ is contained in the base locus of $ \vert m(14\lambda-2\delta) \vert$.

Take $D \in \vert m(14\lambda-2\delta) \vert$ and let $r$ be the multiplicity of $F$ inside $D$.
Consider a Lefschetz pencil of curves of genus
$10$ lying on a general K3 surface of degree $18$ in $\P^{10}$. This gives rise to an irreducible curve $B$ in the moduli space
$\ov{M}_{10}$. Such pencils $B$ fill the divisor $F$, by results of Mukai (\cite{Muk1}).
Therefore $B$ is not contained in the support of $D-rF$, hence $(D-r F)\cdot B\geq 0$.
Using the well-known formulas $\lambda\cdot B=g+1=11$, $\delta_0\cdot B=6(g+3)=78$, $\delta_i\cdot B=0$
for $i\geq 1$ (see e. g. \cite[Lemma 2.1]{FarPop}), together with the expression (\ref{class-K}),
we get that
$$0\leq (D-r F)\cdot B=(2m-r)[(7\lambda-\delta_0)\cdot B]= r-2m,$$
which concludes the proof of the Claim.

>From the previous Claim, it follows that
$$h^0(\ov{M}_{10}, m(14\lambda-2\delta))=h^0(\ov{M}_{10}, m(14\lambda-2\delta)-2mF).$$
Note that $m(14\lambda-2\delta)-2mF=m(\sum_{i\geq 1} a_i\delta_i)$ for some $a_i\geq 0$.
Therefore the proof of the theorem is concluded by the following

\underline{Claim 2}: Let $\Delta$ be an effective divisor in $\mgbar$ (for $g\geq 3$) whose class in $\Pic(\mgbar)_{\Q}$
is equal to $\sum_{i\geq 0} a_i \delta_i$, with $a_i\geq 0$.
Then $h^0(\mgbar, m\Delta)=1$ for any $m$ sufficiently divisible.

Take $E\in \vert m\Delta\vert$. We have to show that $E=m\Delta$.

If $E$ meets the interior $M_g$ of $\mgbar$, then, from the well-known result that $\Pic(M_g)_{\Q}$ is
generated by $\lambda$ and $\lambda$ is ample on $M_g$, we get that the class of
$E$ in $\Pic(\mgbar)_{\Q}$ is equal to $a\lambda+\sum_{i\geq 0}b_i\delta_i$ with $a>0$ and $b_i\in \Z$.
However the class of $E$ in
$\Pic(\mgbar)_{\Q}$ is also equal to the class of $m\Delta$, which is
$\sum_{i\geq 0} m a_i \delta_i$. This produces a non-trivial relation between $\lambda$ and the boundary
classes $\delta_i$, which contradicts the well-known result
that $\Pic(\mgbar)_{\Q}$ is freely generated by $\lambda$ and the boundary classes $\delta_i$ for $g\geq 3$
(see \cite{AC}).

Therefore, $E$ must be entirely contained in the boundary $\mgbar\setminus M_g=\cup_{i\geq 0}\Delta_i$
of $\mgbar$. This implies that $E=\sum_{i\geq 0} b_i\Delta_i$ for some $b_i\geq 0$. Looking at the classes
of $\sum_{i\geq 0} b_i\Delta_i$ and $m\Delta$ in $\Pic(\mgbar)_{\Q}$ and using the independence
of the boundary classes $\delta_i$ in $\Pic(\mgbar)_{\Q}$, we deduce that $E=m\Delta$, as required.
\end{proof}

We finally examine the case $g=11$. As usual, denote by $\F_{g}$ ($g\geq 3$) the moduli space of K3 surfaces
endowed with a polarization of degree $2g-2$. By work of Mukai (\cite{Muk}), there exists a fibration
$$\psi:\ov{M}_{11}\dashrightarrow \F_{11},$$
sending a general curve $C$ of genus $g$ into $(S, \O_S(C))$, where  $S$ is the unique
K3 surface containing $C$.

\begin{Proposition}\label{dim-Kod11}
If $g=11$ then $\kappa(\can)=19$ and the Iitaka fibration of $\can$ is the composition
$$\ov{P}_{d,11}\stackrel{\phi_d}{\longrightarrow} \ov{M}_{11}\stackrel{\psi}{\dashrightarrow}
\F_{11}.
$$
\end{Proposition}
\begin{proof}
Farkas and Popa proved in \cite[Prop. 6.2]{FarPop} that the Iitaka dimension of the divisor
$$E:=7\lambda-\delta_0-5\delta_1-9\delta_2-8\delta_3-7 \delta_4-7 \delta_5$$
is $19$. Since it is easily checked that $14\lambda-2\delta$ is the sum of $2 E$ and an
effective boundary divisor, we get, using (\ref{ineq-Kod}), that
$\kappa(\can)=\kappa(14\lambda-2\delta)\geq \kappa(2E)=\kappa(E)=19$.

Consider now a general point $(S,L)\in \F_{11}$. The fiber of $\psi$ over $(S,L)$
is the open subset of the complete linear series $|L|\cong \P^{11}$ consisting of smooth connected
curves. Pick a Lefschetz pencil on $S$ and consider the associated curve $B$ inside $\ov{M}_{11}$.
It is well-known (see e. g. \cite[Lemma 2.1]{FarPop}) that $\lambda\cdot B=12$, $\delta_0\cdot B=
84$ and $\delta_i\cdot B=0$ for every $i>1$. This easily implies that
\begin{equation}\label{inter-B}
(14\lambda-2\delta)\cdot B=0.
\end{equation}
Consider now the Iitaka fibration of the divisor $K_{\ov{P}_{d,11}}$, which we denote by
$$i_{K_{\ov{P}_{d,11}}}:\ov{P}_{d,11}\dashrightarrow I(K_{\ov{P}_{d,11}}).$$
 Since $K_{\ov{P}_{d,11}}=\phi_d^*(14\lambda-2\delta)$,
the Iitaka fibration $i_{14\lambda-2\delta}$ of $K_{\ov{P}_{d,11}}$ is the composition of the Iitaka fibration of
$14\lambda -2\delta$ with $\phi_d$, i.e. we have a natural diagram (up to birationality)
$$\xymatrix{
\ov{P}_{d,11}\ar[rr]^{\phi_d} \ar@{-->}[dr]^{i_{K_{\ov{P}_{d,11}}}}& &
\ov{M}_{11}\ar@{-->}[dl]_{i_{14\lambda-2\delta}} \\
& I(K_{\ov{P}_{d,11}})= I(14\lambda-2\delta)
}$$
Now, equation (\ref{inter-B}) implies that the Iitaka fibration $i_{14\lambda-2\delta}$ contracts
the general fiber $\psi^{-1}(S,L)\subset |L|\cong \P^{11}$. Therefore the Iitaka fibration
$i_{14\lambda-2\delta}$ factors through the fibration $\psi$:
$$\xymatrix{
& \ov{M}_{11} \ar@{-->}[dr]^{i_{14\lambda-2\delta}} \ar@{-->}[dl]_{\psi} & \\
\F_{11} \ar@{-->}[rr]^{\rho} && I(14\lambda-2\delta)
}$$
Recall that $\dim \F_{11}=19$. On other hand, by the usual properties of the Iitaka fibration and
what proved before, we have that $\dim I(14\lambda-2\delta)=\kappa(14\lambda-2\delta)\geq 19$.
Since $\rho$ is dominant and has connected general fiber, this implies that $\rho$ is
a birational isomorphism, hence we are done.
\end{proof}

\section{Birationalities among different $P_{d,g}$'s}\label{bir-Pdg}

In this section, inspired by Lemma 8.1 in \cite{Cap}, we investigate the following

\begin{Question}\label{birat-que}
For what values of $d$ and $d'$ is $P_{d,g}$ birational to $P_{d',g}$?
How do the birational maps among them look like?
\end{Question}

Note that if $d'=d+n(2g-2)$ for some $n\in \Z$ then we have the isomorphism
\begin{equation}\label{bira1}
\begin{aligned}
\psi^1_n:P_{d,g} & \stackrel{\cong}{\longrightarrow} P_{d',g} \\
(C,L) & \mapsto (C,L\otimes \omega_C^n),
\end{aligned}
\end{equation}
while if $d'=-d+n(2g-2)$ for some $n\in \Z$ we have the isomorphism
\begin{equation}\label{bira2}
\begin{aligned}
\psi^2_n:P_{d,g} & \stackrel{\cong}{\longrightarrow} P_{d',g} \\
(C,L) & \mapsto (C,L^{-1}\otimes \omega_C^n).
\end{aligned}
\end{equation}
Clearly the maps $\psi_n^1$ and $\psi_n^2$ commute with the projections $\phi_d$ and $\phi_{d'}$ onto $M_g$. Indeed, the converse is also true,
as it follows from an argument of Caporaso  (see \cite[Lemma 8.1]{Cap} and also, for further details,  \cite[Prop. 3.2.2]{Cap-Lis}).
%for a more detailed proof).

%\noindent Caporaso proved the following partial result

\begin{Theorem}\emph{(\cite{Cap})}\label{Cap-thm}
If $\eta: P_{d,g}\dashrightarrow P_{d',g}$ is a birational map over $M_g$, i.e. a map $\eta$ inducing a commutative diagram
$$\xymatrix{
P_{d,g}\ar@{-->}[r]^{\eta}\ar[d]_{\phi_d}& P_{d',g}\ar[d]^{\phi_{d'}} \\
M_g\ar[r]^{{\rm id}} & M_g
}
$$
then there exists $n\in \Z$ such that
either $d'=d+n(2g-2)$ and  $\eta=\psi^1_n$ or $d'=-d+n(2g-2)$ and $\eta=\psi^2_n$.
\end{Theorem}

By using our results on the Kodaira dimension of $P_{d,g}$, we can improve Theorem \ref{Cap-thm} at least for genus big enough.

%We prove that the above morphisms are the unique birational maps between the $P_{d,g}$:

\begin{Theorem}\label{birationa}
Assume that $g\geq 22$ or $g\geq 12$ and $(d-g+1,2g-2)=1$. Let $\eta: P_{d,g}\dashrightarrow P_{d',g}$ be a birational map.
Then there exists $n\in \Z$ such that either $d'=d+n(2g-2)$ and  $\eta=\psi^1_n$ or $d'=-d+n(2g-2)$ and $\eta=\psi^2_n$.
In particular, $\eta$ is an isomorphism.
\end{Theorem}

\begin{proof}
%Let $\eta:P_{d,g} \dashrightarrow P_{d',g}$ be a birational map. The restriction of the  composition
%$$P_{d,g}\stackrel{\eta}{\dashrightarrow}P_{d',g}\stackrel{\phi_{d'}}{\rightarrow}M_g$$
%to the general fiber $\phi_d^{-1}([C])=\Pic^d(C)$
%of $\phi_d$ is a rational map $\nu:\Pic^d(C)\dashrightarrow M_g$. By resolving the rational map $\nu$,
%we obtain a diagram
%$$\xymatrix{
%& \w{\Pic^d(C)}\ar[dl]_b \ar[dr]^{\w{\nu}} & \\
%\Pic^d(C) \ar@{-->}[rr]^{\nu} && M_g
%}$$
%where $b$ is a composition of blow-ups at smooth centers. Since $C$ is generic, the map $\w{\nu}$ takes
%values in the open subset  $M_g^0\subset M_g$ consisting of curves without non-trivial automorphisms for $g\geq 3$
%and curves with automorphism group $\Z/2\Z$ for $g=2$. By \cite{VZ}, $M_g^0$ is Brody-hyperbolic, i.e. all the entire curves
%$\C\to M_g^0$ are constant. On the other hand, the variety $\w{\Pic^d(C)}$,
%being an iterated blow-up of an abelian variety, has the property
%that any two general points on it are connected by an entire curve.
% (an entire curve being a holomorphic map $\C\to \w{\Pic^d(C)}$).
%We deduce that $\w{\nu}$ (hence a fortiori $\nu$ as well) is constant.

By Theorems \ref{Kod-geom} and \ref{kod-nongeo}, the assumptions of the statement imply that $\kappa(P_{d,g})=3g-3$, hence
that $\kappa(P_{d',g})=3g-3$. From the proof of Proposition \ref{kod-fib}, it follows that $\phi_d:P_{d,g}\to M_g$ is the Iitaka fibration of
$P_{d,g}$ and similarly for $P_{d',g}$. Since the Iitaka fibration is a birational invariant,
%This implies that $\eta$ sends the generic fiber of $\phi_d$ to the generic fiber of $\phi_{d'}$. In other words,
the map $\eta$ induces a birational map
$\xi:M_g\dashrightarrow M_g$ such that the following diagram commutes:
$$\xymatrix{
P_{d,g} \ar@{-->}[r]^{\eta} \ar[d]_{\phi_d}& P_{d',g}\ar[d]^{\phi_{d'}} \\
M_g \ar@{-->}[r]^{\xi} & M_g.
}$$
The map $\xi$  sends a very general curve $C\in M_g$ to a very general curve $C'\in M_g$ so that
the restriction of $\eta$  induces a birational map $J(C)\cong\Pic^d(C) \dashrightarrow \Pic^{d'}(C')\cong J(C')$.
By Lemma \ref{gen-Tor} below, we get that  $C\cong C'$, hence $\xi={\rm id}$.
Finally, we conclude by  Theorem \ref{Cap-thm}.
\end{proof}

\begin{Lemma}\label{gen-Tor}
If $C$ and $C'$ are very general curves in $M_g$ such that $J(C)$ is birational to $J(C')$, then $C\cong C'$.
\end{Lemma}
\begin{proof}
Let $\epsilon:J(C)\dashrightarrow J(C')$ be a birational map. Since $J(C)$ and $J(C')$ are abelian varieties, then it is well-known
that $\epsilon$ extends to an isomorphism $\epsilon:J(C)\stackrel{\cong}{\to} J(C')$.
Since $C$ (resp. $C'$) are very general curves in $M_g$, we may assume (by \cite[Cor. 17.5.2]{BL}) that $\NS(J(C))=\Z$ (resp. $\NS(J(C'))=\Z$)  generated
by the class of the theta divisor $[\Theta_C]$ (resp. $[\Theta_{C'}]$). Therefore we must have
$\epsilon^*([\Theta_{C'}])=\pm [\Theta_C]$. Moreover, since $[\Theta_C]$ is ample and the pull-back morphism $\epsilon^*$
preserves ampleness, we get that actually $\epsilon^*([\Theta_{C'}])=[\Theta_C]$.
%By possibly composing $\epsilon$ with the involution on $J(C')$, we can assume that $\epsilon^*([\Theta_{C'}])= [\Theta_C]$.
We conclude that $C\cong C'$ by the classical Torelli theorem.
\end{proof}

>From the previous result, we can deduce two corollaries. The first one concerns the group of birational self maps ${\rm Bir}(P_{d,g})$
and the group of automorphisms ${\rm Aut}(P_{d,g})$ of $P_{d,g}$.

\begin{Corollary}\label{Cor-bira1}
With the same assumptions as in Theorem \ref{birationa},
%Assume that $g\geq 22$ or $g\geq 12$ and $(d-g+1,2g-2)=1$. Then
we have
$${\rm Bir}(P_{d,g})={\rm Aut}(P_{d,g})=
\begin{cases}
\Z/2\Z & \text{ if } d(g-1) \text{ for some } n\in \Z,\\
\{{\rm id}\} & \text{ otherwise,}
\end{cases}
$$
where in the first case the generator of the cyclic group $\Z/2\Z$ is $\psi_n^2$.
\end{Corollary}
\begin{proof}
The claim follows from Theorem \ref{birationa}, since the only maps $\psi_n^1$ and $\psi_n^2$ having domain $P_{d,g}$ and codomain $P_{d,g}$ are
$\psi_0^1={\rm id}$ and $\psi_n^2$ if $d(g-1)$, and in this case $(\psi_n^2)\circ (\psi_n^2)={\rm id}$.
\end{proof}

The second corollary is analogous to \cite[Cor. (0.12)]{GKM}, which states that the boundary $\partial \ov{M}_{g}$ of the moduli space $\ov{M}_g$ of stable curves of genus
$g\geq 2$ is preserved by any automorphism of $\ov{M}_g$.

\begin{Corollary}\label{Cor-bira2}
Same assumptions as in Theorem \ref{birationa}.
%Assume that $g\geq 22$ or $g\geq 12$ and $(d-g+1,2g-2)=1$.
Then any automorphism $\phi:\ov{P}_{d,g}\to \ov{P}_{d,g}$ preserves the boundary
$\partial \ov{P}_{d,g}:=\ov{P}_{d,g}\setminus P_{d,g}$.
\end{Corollary}
\begin{proof}
The restriction $\eta:=\phi_{|P_{d,g}}$ of $\phi$ to $P_{d,g}$ defines a birational self map of $P_{d,g}$. By Corollary \ref{Cor-bira1}, $\eta$
is an automorphism of $P_{d,g}$. Therefore $\phi$ maps $P_{d,g}$ isomorphically onto $P_{d,g}$, hence it preserves the boundary $\partial \ov{P}_{d,g}$.
\end{proof}

\begin{Remark}\label{Aut-comp}
Under the same assumptions as in Theorem \ref{birationa}, Corollary \ref{Cor-bira2} implies that we have a restriction map
$${\rm res}: \Aut(\pdgbar)\to \Aut(\pdg).$$
The map ${\rm res}$ is injective since $\pdgbar$ is separated. In \cite[Lemma 8.1]{Cap}, it is claimed
%(without a proof)
 that the map
$\psi_n^2$ of \eqref{bira2} extends to a map $\pdgbar\to \ov{P}_{d',g}$  (the analogous statement for $\psi_n^1$ is easy to prove).
%$\Aut(\pdg)$ extends to an automorphism of $\pdgbar$.
This fact, together with Corollary \ref{Cor-bira1}, would imply that ${\rm res}$ is an isomorphism.
\end{Remark}
%\begin{Remark}\label{Rmk-genus}
Finally, note that if one could remove our technical assumption on the degree in Theorem \ref{Kod-geom}, then Theorem \ref{birationa} and Corollaries \ref{Cor-bira1} and \ref{Cor-bira2}
would follow for $g\geq 12$ without any hypothesis on the degree.
%However, we don't known what happens for small values of $g$.
%\end{Remark}

\section*{Acknowledgements}
We are grateful to Silvia Brannetti, Lucia Caporaso, Eduardo Este\-ves and Margarida Melo for stimulating discussions on these topics.
We thank Gavril Farkas and Sandro Verra for pointing out a gap in the proof of a previous version of Theorem \ref{birationa}.
We thank the two referees of the paper for their precious remarks that helped in improving the exposition.
In particular, we are grateful to one of the referees for pointing out a mistake in a previous version of Theorem \ref{conj-sing} and of Theorem \ref{sing-Pdst}.
% this actually led to a simplification of the argument.

\end{document}